\documentclass[final]{siamltex}

\title{Small-gain theorems for nonlinear stochastic systems with inputs and outputs II: Multiplicative white noise case\thanks{This work was supported by the National Natural Science Foundation of China (NSFC) under
Grants No.11371252 and No.11501369; Research and Innovation Project of Shanghai Education Committee under Grant No.14zz120; Chen Guang Project(14CG43) of Shanghai Municipal Education Commission, Shanghai Education Development Foundation; Yangfan Program of Shanghai (14YF1409100); the Research Program of Shanghai Normal University (SK201403) and Shanghai Gaofeng Project for University Academic Program Development.}}

\author{Jifa Jiang\thanks{Department of Mathematics, Shanghai Normal
University, Shanghai 200234, PR China ({\tt jiangjf@shnu.edu.cn}).}
        \and Xiang Lv\thanks{Corresponding author. Department of Mathematics, Shanghai Normal
University, Shanghai 200234, PR China ({\tt lvxiang@shnu.edu.cn}).}}

\usepackage[leqno]{amsmath}
\usepackage{amsmath,amssymb,amsfonts}
\usepackage{color,xcolor}
\usepackage{graphicx}
\usepackage{manfnt}
\usepackage[all]{xy}
\usepackage{enumerate}
\usepackage{mathrsfs}

\begin{document}

\maketitle

\begin{abstract}
This paper is a continuation of the paper \cite{JL},  which focuses on exploring the global stability of nonlinear stochastic feedback systems on the nonnegative orthant driven by multiplicative white noise and presenting a couple of small-gain results. We investigate the dynamical behavior of pull-back trajectories for stochastic control systems and prove that there exists a unique globally attracting positive random equilibrium for those systems whose output functions either possess bounded derivatives or are uniformly bounded away from zero. In the first case, we first prove the joint measurability of both the pull-back trajectories and the metric dynamical system $\theta$ with respect to the product $\sigma$-algebra $\mathscr{B}(\mathbb{R_+})\otimes\mathscr{F}_-$ and $\mathscr{B}(\mathbb{R_-})\otimes\mathscr{F}_-$, respectively, where $\mathscr{F}_-=\sigma\{\omega\mapsto W_t(\omega):t\leq0\}$ is the past $\sigma$-algebra and $W_t(\omega)$ is an $\mathbb{R}^d$-valued two-sided Wiener process, and then combine the $\mathcal{L}^1$-integrability of the tempered random variable coming from the definition of the top Lyapunov exponent and the independence between the past $\sigma$-algebra and the future $\sigma$-algebra $\mathscr{F}_+=\sigma\{\omega\mapsto W_t(\omega):t\geq0\}$ to obtain the small-gain theorem by constructing the contraction mapping on an $\mathscr{F}_-$-measurable, $\mathcal{L}^1$-integrable and complete metric input space; in the second case, the sublinearity of  output functions and the part metric play the main roles in the existence and uniqueness of globally attracting positive fixed point in the part of a normal, solid cone. Our results can be applied to well-known stochastic Goodwin negative feedback system, Othmer-Tyson positive feedback system and Griffith positive feedback system as well as other stochastic cooperative, competitive and predator-prey systems.
\end{abstract}



\pagestyle{myheadings}
\thispagestyle{plain}

\section{Introduction}
The purpose of this paper is to consider the stability of nonlinear stochastic control systems with multiplicative white noise, where stochastic stability results can be regarded as a small-gain theorem for guaranteeing the existence, uniqueness and global attractivity of positive random equilibrium for the corresponding random dynamical systems. It is well known that small-gain theory is an important tool to investigate the behavior of many linear or nonlinear control systems.  The history of small-gain theorems can be traced  back to Zames \cite{Z} in 1966, who proved that control system admits input-output stability if the (incremental) gain-product is less than one, that is, the relations (mappings) between inputs and outputs are bounded and continuous.  Motivated by the excellent work of Zames \cite{Z}, many small-gain theorems for various feedback systems were proposed, see \cite{AA,ALS,DRW,ES,H,JM,JMW,JTP,LN,MH,NT,S,SI,ST} and references therein, which have been widely used in the stability analysis and design of many interconnected systems.

Control theory is an interdisciplinary field of engineering, biology, chemistry and mathematics, which was originally developed to provide tools for analysis of control systems.
During the past development, it has been widely applied to investigate the dynamical behavior of deterministic control systems with inputs and outputs. Regarded as feedback models, especially those from biology and ecology et al.,
they are often subjective to noise disturbances either from their surrounding environments or from their intrinsic uncertainties. Noise-perturbed systems for continuous-time are usually described by stochastic differential equations. Studying stochastic stability for stochastic control systems is very important both in theory and in practice, see \cite{DK1,DK2,DK3,PB}. Among the existing works on stochastic stability, most of them focus on the trivial solution, that is, stochastic stability for the zero solution in respective meanings, see \cite{M} and its references. However, if the zero solution, or in general a deterministic function, is not a solution of given stochastic control system, then the most suitable candidate to replace the trivial solution is a stationary solution. But to study the existence and global stability for a stationary solution in certain sense is not an easy job, there are very few works on this subject as far as we know.  Therefore, exploring stochastic stability for a nontrivial stationary solution of stochastic control system is a challenging problem.

Recently, the theory of random dynamical systems has been developed to investigate the stability of control systems involving in {\it real noise} perturbation by Marcondes de Freitas and Sontag \cite{FS1,FS2,FS3}, which provide  new insights to consider non-monotone systems. Specifically, they established a small-gain theorem to characterize interconnections and proved the existence of a globally attractive random equilibrium for random systems with inputs and outputs by iterating the constructed $``gain\ operator"$  $\mathcal{K}^{h}(u)$, this random equilibrium is usually nontrivial and only qualitatively exists. It is worth noting that their main result can not be applied to stochastic systems driven by Brownian motion with inputs and outputs. Although Marcondes de Freitas and Sontag's small-gain Theorem 4.4 in \cite{FS3} is general in form, it seems to us that their setting was based on continuous-time/discrete-time random systems. Almost all they considered are $\omega$-wised without involving in any stochastic integral, adaption et al. If we follow their streamline for stochastic control systems driven by Brownian motion, we will encounter many difficulties to overcome among which the perfection of the crude
cocycle $\varphi(t,\omega,x,u)$ for nonautonomous stochastic systems with input $u$ is the biggest one. However, the pivotal role of the small-gain condition in \cite{FS3} is to guarantee the existence, uniqueness of fixed point $u$  for $\mathcal{K}^{h}$ and global attractivity for $\mathcal{K}(u)$.  The present authors  \cite{JL} considered the global stability of nonlinear stochastic control systems driven by additive white noise. We came quickly and directly to the problem with just the necessary minimum of technical preparations. Precisely,  we first defined the {\it ``input-to-state characteristic operator"} $\mathcal{K}(u)$ of the system in a suitably chosen input space via the backward It\^{o} integral, and then for a given output function $h$, defined the $``gain\ operator"$  $\mathcal{K}^{h}(u)$ as the composition of the output function $h$ and the input-to-state characteristic operator $\mathcal{K}(u)$ on the input space. Suppose that the output function is either order-preserving or anti-order-preserving in the usual vector order and the global Lipschitz constant of the output function is less than the absolute of the negative principal eigenvalue of linear matrix. Then we proved a small-gain theorem via the Banach fixed point theorem, which is much more easily accessible.

 It is more realistic to restrain the state space on the nonnegative orthant when we consider the feedback problems originating from biology, ecology and biochemistry et al. For these problems, we have to consider stochastic systems driven by multiplicative white noise, which usually preserves the invariance for solutions on the nonnegative orthant. This paper is a continuation of the paper \cite{JL} and presents a couple of small-gain results for nonlinear stochastic control systems on the nonnegative orthant driven by multiplicative white noise. Using the ideas in \cite{JL}, we first define the input-to-state characteristic operator $\mathcal{K}(u)$ via the stochastic system with inputs (see (\ref{input})) and the backward It\^{o} integral and then define the gain operator $\mathcal{K}^{h}(u):=h\circ \mathcal{K}(u)$. Our task is to construct a suitable complete metric (measurable) input  space such that $\mathcal{K}^{h}(\cdot)$ is contractive on it. Throughout of this paper, we assume that the matrix $A$ in the stochastic system with inputs is cooperative and the fundamental matrix of the corresponding homogeneous system for the stochastic system with inputs has a negative top Lyapunov exponent and that the output function is either monotone or anti-monotone in usual vector field order. Two small-gain theorems will be verified  in the two cases that output functions are either uniformly bounded away from zero or possess bounded derivatives. Noting that for the case of multiplicative white noise, the fundamental matrix $\Phi(t,\omega)$ depends on $\omega\in\Omega$. This yields that we can not give uniform estimate of $\Phi(t,\omega)$ for all $\omega\in\Omega$ and the problem will become more difficult to study than that in additive white noise \cite{JL}. To overcome this difficulty, in the case that derivatives of output functions are bounded, we first establish the subtle joint measurability of the metric dynamical system $\theta$ and pull-back trajectories with respect to the product $\sigma$-algebra $\mathscr{B}(\mathbb{R_-})\otimes\mathscr{F}_-$ and $\mathscr{B}(\mathbb{R_+})\otimes\mathscr{F}_-$, respectively, where $\mathscr{F}_-=\sigma\{\omega\mapsto W_t(\omega):t\leq0\}$, $W_t(\omega)$ is an $\mathbb{R}^d$-valued two-sided Brownian motion.  These measurabilities help us to well define the input-to-state characteristic operator $\mathcal{K}(u)$ and the gain operator $\mathcal{K}^h(u)$ on a bounded and $\mathscr{F}_-$ measurable input space. Then
making use of the $\mathcal{L}^1$-integrability of the tempered random variable in the definition of the top Lyapunov exponent and the independence between this tempered random variable and  the past $\sigma$-algebra $\mathscr{F}_-$, we prove  that the gain operator $\mathcal{K}^{h}(\cdot)$ is contractive on this input space and carry out the ideas in  \cite{JL} to achieve in proving a small-gain theorem.  In the case that output functions are uniformly bounded away from zero, the main tool presented here is to take advantage of the sublinearity of output functions and the part metric for guaranteeing the existence and uniqueness of a globally attracting fixed point in the part of a normal, solid cone. Our results can be successfully applied to well-known stochastic Goodwin negative feedback system, Othmer-Tyson positive feedback system and Griffith positive feedback system as well as other stochastic cooperative, competitive and predator-prey systems to get that there exists a unique globally stable positive random equilibrium on the nonnegative orthant.

This paper is organized as follows. Section 2 contains the considered problem, some preliminary definitions, notations and the definition for the input-to-state characteristic map. Section 3 specifies the measurability of pull-back trajectories and the metric dynamical system $\theta$ and the dynamical behavior of stochastic flows, and shows some relative order-preserving results. Section 4 presents two stochastic small-gain theorems, which are  applied to a series of examples. Section 5 ends this paper with some concluding remarks and discussions.

\section{Problem and preliminaries}
In this section, we will investigate a stochastic biochemical model  consisting of $d$ interacting components, which may be more realistic for describing and simulating the dynamical behavior of biochemical networks under fluctuations of intrinsic and extrinsic noise. Let $X_i$ represent the $i$-th variable of biochemical reactions (protein concentrations or levels of gene expression), which can be modelled by the following nonlinear stochastic system with multiplicative white noise
\begin{equation}
dX_t=[AX_t+h(X_t)]dt+\sum_{k=1}^{d}\sigma_kX_tdW_t^k,\label{problem}
\end{equation}
where $A=(a_{ij})_{d\times d}$ and $\sigma_k=(\sigma_k^{ij})_{d\times d}$ are ($d\times d$)-dimensional matrices, $k=1,\ldots,d$, $h:\mathbb{R}^d_+\rightarrow\mathbb{R}^d_+$ and
$W_t(\omega)=\left(W_t^1(\omega),\ldots,W_t^d(\omega)\right)$ is an $\mathbb{R}^d$-valued two-sided Wiener process on the probability space $(\Omega,\mathscr{F},\mathbb{P})$, where $\mathscr{F}$ is the Borel $\sigma$-algebra of
$\Omega=C_0(\mathbb{R},\mathbb{R}^d)=\{\omega=(\omega_1,\omega_2,\ldots,\omega_d)\in C(\mathbb{R},\mathbb{R}^d),\ \omega(0)=0\}$
induced by the compact-open topology, which is generated by the following metric
\[\varrho(\omega,\omega^\ast):=\sum_{n=1}^\infty\frac{1}{2^n}\frac{\varrho_n(\omega,\omega^\ast)}{1+\varrho_n(\omega,\omega^\ast)},\quad
\varrho_n(\omega,\omega^\ast)=\max_{t\in[-n,n]}|\omega(t)-\omega^\ast(t)|,\]
and $\mathbb{P}$ is the corresponding Wiener measure. Furthermore, we will be concerned with the dynamical behavior of stochastic differential equations on the nonnegative orthant. For this purpose, we assume that $\sigma_k$, $k=1,\ldots,d$, has the following form throughout this paper
\begin{equation}
\sigma_k=\left[\begin{array}{ccc} \sigma_k^{1}& &\\
 &\ddots& \\
& & \sigma_k^d\end{array}\right],\qquad \sigma_k^i\in\mathbb{R},\ k,i=1,\ldots,d.
\label{eq1}
\end{equation}

For the convenience of readers, we will give some definitions and notations of random dynamical systems for later use, see \cite{A,C} for more details.

In this work, let $X$ be a Polish space endowed with the Borel $\sigma$-algebra $\mathscr{B}(X)$, i.e., a separable complete metric space, and $(\Omega,\mathscr{F},\mathbb{P})$ a probability space.

\begin{definition}
$\theta\equiv(\Omega,\mathscr{F},\mathbb{P},\{\theta_t,t\in\mathbb{R}\})$ is called a metric dynamical system if
\begin{enumerate}[{\rm (i)}]
\item $\theta:\mathbb{R}\times\Omega\mapsto\Omega$ is $(\mathscr{B}(\mathbb{R})\otimes\mathscr{F}, \mathscr{F})$-measurable;
\item $\theta_0={\rm id}$ is the identity on $\Omega$ and $\theta_t\circ\theta_s=\theta_{t+s}$ for all $t,s\in\mathbb{R}$;
\item $\theta_t\mathbb{P}=\mathbb{P}$ for all $t\in\mathbb{R}$, i.e.,
$\theta_t$ preserves the probability measure $\mathbb{P}$ for all $t\in\mathbb{R}$.
\end{enumerate}
\end{definition}

\begin{definition}
A random dynamical system (RDS) on the Polish space $X$ consists of two elements: a metric dynamical
system $\theta\equiv(\Omega,\mathscr{F},\mathbb{P},\{\theta_t,t\in\mathbb{R}\})$ and the mapping
\[\varphi:\mathbb{R}_+\times\Omega\times X\mapsto X, \quad (t,\omega,x)\mapsto\varphi(t,\omega,x),\]
which is $(\mathscr{B}(\mathbb{R}_+)\otimes\mathscr{F}\otimes\mathscr{B}(X),
\mathscr{B}(X))$-measurable and satisfies
\begin{enumerate}[{\rm (i)}]
\item$\varphi(t,\omega,\cdot): X \to X$ is continuous for all $t\in
\mathbb{R}_+$ and $\omega\in\Omega$;
\item the mappings $\varphi(t,\omega):=\varphi(t,\omega,\cdot)$ satisfy the cocycle (over $\theta$) property:
\[\varphi(0,\omega)={\rm id},\quad \varphi(t+s,\omega)=\varphi(t,\theta_s\omega)\circ\varphi(s,\omega)\]
for all $t,s\in\mathbb{R}_+$ and $\omega\in\Omega$.
\end{enumerate}
\end{definition}

\begin{definition}
The multifunction $D:\Omega\to
2^X\backslash\{\varnothing\}$ is said to be a random set if the mapping
$\omega\to {\rm dist}_X(x,D(\omega))$ is measurable for any $x\in X$, where ${\rm dist}_X(x,B)$ means the
distance in $X$ between the point $x$ and the set $B\subset X$. If $D(\omega)$ is closed (resp. compact) in $X$ for each $\omega\in\Omega$, the mapping $\omega\to D(\omega)$ is called a random closed (resp. compact) set.
\end{definition}

Motivated by the work of \cite{FS3,JL}, the above stochastic model can be rewritten as a stochastic system with inputs
\begin{equation}
dX_t=[AX_t+u(t)]dt+\sum_{k=1}^{d}\sigma_kX_tdW_t^k
\label{input}
\end{equation}
together with outputs
\[u(t)=h(X_t).\]
From this viewpoint, we can regard the nonlinear feedback function $u(t)=h(X_t)$ as a known stochastic process, which results in that the stochastic system (\ref{problem}) will become a linear non-homogeneous stochastic differential equations.

Let us first consider the corresponding linear homogeneous stochastic It${\rm  \hat{o}}$ type differential equations
\begin{equation}
dX_t=AX_tdt+\sum_{k=1}^{d}\sigma_kX_tdW_t^k,
\label{eq2}
\end{equation}
which is equivalent to the following system of Stratonovich stochastic differential equations
\begin{equation}
dX_t=(A-\frac12C)X_tdt+\sum_{k=1}^{d}\sigma_kX_t\circ dW_t^k,
\label{eq3}
\end{equation}
where we write $C$ in the form
\begin{equation}
C=\left[\begin{array}{ccc} \sum\limits_{k=1}^{d}(\sigma_k^1)^2& &\\
 &\ddots& \\
& & \sum\limits_{k=1}^{d}(\sigma_k^d)^2\end{array}\right].\label{eq4}
\end{equation}

In order to make use of the technique of monotone systems, it is necessary to make the following assumption on $A$
\begin{enumerate}[({A})]
\item
$A$ is {\it cooperative}, i.e.,
$a_{ij}\geq0$ for all $i,j\in\{1,\ldots,d\}$ and $i\neq j$.
\end{enumerate}

We will denote by $\Phi(t)=(\Phi_1(t),\ldots,\Phi_d(t))=(\Phi_{ij}(t))_{d\times d}$ the fundamental matrix of equations (\ref{eq2}), where $\Phi_j(t)=(\Phi_{1j}(t),\ldots,\Phi_{dj}(t))^T$ is the solution of equations (\ref{eq2}) with initial value $x(0)=e_j$, $j=1,\ldots,d$.
By the classical existence and uniqueness of solutions for stochastic differential equations and the theory of monotone random dynamical systems, it is clear that (\ref{eq2}), i.e., (\ref{eq3}) generates a linear order-preserving random dynamical system $(\theta,\Phi)$ in $\mathbb{R}^d_+$, see \cite[Proposition 6.2.2, p. 186]{C}, where $\theta$ is the time shift on $\Omega$, i.e.,
\[\theta_t\omega(\cdot):=\omega(t+\cdot)-\omega(t),\quad t\in\mathbb{R}.\]
That is, $\Phi$ satisfies the cocycle property: $\Phi(t+s,\omega)=\Phi(t,\theta_s\omega)\circ\Phi(s,\omega)$
for all $t,s\in\mathbb{R}_+$, $\omega\in\Omega$, and $\Phi(t,\omega)x\geq_{\mathbb{R}^d_+}\Phi(t,\omega)y$, for all $x,y\in\mathbb{R}^d_+$ such that $x\geq_{\mathbb{R}^d_+}y$, where $x\geq_{\mathbb{R}^d_+}y$ means that $x-y\in\mathbb{R}^d_+$. Furthermore, the following assumption on $(\theta,\Phi)$ will be needed in what follows
\begin{enumerate}[({L})]
\item
The top Lyapunov exponent for the linear RDS $(\theta,\Phi)$ is a negative real number, i.e., there exist a constant $\lambda>0$ and a
tempered random variable $R(\omega)>0$ such that
\begin{equation}
\|\Phi(t,\omega)\|:=\max\{|\Phi_{ij}(t,\omega)|:i,j=1,\ldots,d\}\leq R(\omega)e^{-\lambda t}
\label{eq5}
\end{equation}
holds for all $t\geq0$, $\omega\in\Omega$.
\end{enumerate}
Here, a random variable $R(\omega)>0$ is called tempered if
\[\sup_{t\in\mathbb{R}}\left\{e^{-\gamma|t|}\left|R(\theta_t\omega)\right|_2\right\}<\infty \quad {\rm for\ all}\ \omega\in\Omega  \ {\rm and}\ \gamma>0,\]
where $|x|_2:=(\sum_{i=1}^d|x_i|^2)^{\frac12}$, $x\in\mathbb{R}^d$. Throughout this paper, we will use the norm $|x|:= \max\{|x_i|: i=1,\ldots,d\}$, $x\in\mathbb{R}^d$ and $\|\Phi\|_2:=(\sum_{i,j=1}^d|\Phi_{ij}|^2)^{\frac12}$, $\Phi\in\mathbb{R}^{d\times d}$.

In the remainder of this section, we are concerned with the existence and uniqueness of solutions for (\ref{problem}) and its pull-back trajectories, we shall make the following assumption on $h$, which is abstracted from the Othmer-Tyson positive feedback model \cite{TO} and the Goodwin negative feedback model \cite{G}.
\begin{enumerate}[({H}$_1$)]
\item
$h\in C^1(\mathbb{R}^d_+,\mathbb{R}^d_+\setminus\{0\})$ and is bounded in $\mathbb{R}^d_+$. Moreover, we assume that $h$ is monotone, i.e.,
    \[x_1\leq_{\mathbb{R}^d_+}x_2\quad\Rightarrow\quad h(x_1)\leq_{\mathbb{R}^d_+}h(x_2),\qquad \forall x_1, x_2\in\mathbb{R}^d_+\]
or anti-monotone, i.e.,
\[x_1\leq_{\mathbb{R}^d_+}x_2\quad\Rightarrow\quad h(x_1)\geq_{\mathbb{R}^d_+}h(x_2),\qquad \forall x_1, x_2\in\mathbb{R}^d_+.\]
\end{enumerate}
By (H$_1$), it is easy to check that (\ref{problem}) satisfies the conditions of local Lipschitz and linear growth (since $h$ is bounded in $\mathbb{R}^d_+$) in $\mathbb{R}^d_+$. Motivated by the proof of Proposition 6.2.1 in \cite{C}, let $\tilde{h}$ be an extension of $h$ from $\mathbb{R}^d_+$ to $\mathbb{R}^d$ such that $\tilde{h}$ satisfies the conditions of local Lipschitz and linear growth in $\mathbb{R}^d$, we thus have the existence and uniqueness of global solutions for
\[dX_t=[AX_t+\tilde{h}(X_t)]dt+\sum_{k=1}^{d}\sigma_kX_tdW_t^k,\]
see \cite{M,O}, which is equivalent to the Stratonovich interpretation of stochastic differential equations
\[dX_t=[(A-\frac12C)X_t+\tilde{h}(X_t)]dt+\sum_{k=1}^{d}\sigma_kX_t\circ dW_t^k,\]
where $C$ is defined in (\ref{eq4}), and generates an RDS in $\mathbb{R}^d$, see \cite[Chapter 2]{A,C}. In the same manner of Proposition 6.2.1 in \cite{C}, we can see that there exists a unique (indistinguished) RDS $(\theta,\varphi)$ generated by (\ref{problem}) such that the set  $\mathbb{R}^d_+$ is forward invariant, i.e., $\varphi(t,\omega)\mathbb{R}^d_+\subset\mathbb{R}^d_+$ for all $t\in\mathbb{R}_+$, $\omega\in\Omega$ and $\varphi(t,\omega)x=x(t,\omega,x)$ is the unique solution of equations (\ref{problem}) for each initial value $x(0)=x\in\mathbb{R}^d_+$.

Combining the variation-of-constants formula \cite[Chapter 3, Theorem 3.1]{M} and the cocycle property of $\Phi$, it follows that
\begin{eqnarray}\label{eq6}
\varphi(t,\omega)x &=& \Phi(t,\omega)x+\Phi(t,\omega)\int_0^t\Phi^{-1}(s,\omega)h(\varphi(s,\omega)x)ds\nonumber\\
&=& \Phi(t,\omega)x+\int_0^t\Phi(t-s,\theta_s\omega)h(\varphi(s,\omega)x)ds, \quad  t\geq0,\ \omega\in\Omega.
\end{eqnarray}

By the definition of $\theta$, a similar analysis as in \cite{JL} shows that the pull-back trajectories of $(\theta,\varphi)$ as follows
\begin{eqnarray}\label{eq7}
\varphi(t,\theta_{-t}\omega)x &=& \Phi(t,\theta_{-t}\omega)x+\int_0^t\Phi(t-s,\theta_{s-t}\omega)h(\varphi(s,\theta_{-t}\omega)x)ds\nonumber\\
&=& \Phi(t,\theta_{-t}\omega)x+\int_{-t}^0\Phi(-s,\theta_s\omega)h(\varphi(t+s,\theta_{-t}\omega)x)ds,\ t\geq0,\ \omega\in\Omega.
\end{eqnarray}

Regarding the feedback function $h$ as an input term, we define the {\it input-to-state characteristic} map $\mathcal{K}$ associated with given inputs in $\mathbb{R}^d_+$ as follows
\begin{equation}\label{eq8}
[\mathcal{K}(u)](\omega)=\int_{-\infty}^0\Phi(-s,\theta_s\omega)u(\theta_s\omega)ds,\quad\omega\in\Omega,
\end{equation}
where $u$ is an $\mathbb{R}^d_+$-valued and tempered random variable with respect to $\theta$.

\textsc{{\it Remark}} 1. Since the top Lyapunov exponent of $\Phi$ is negative, it is evident that $\mathcal{K}$ is well defined. We first observe the fact that
$\|\Phi(t,\omega)\|_2\leq d\|\Phi(t,\omega)\|\leq dR(\omega){\rm e}^{-\lambda t}$, $\lambda>0$. According to the above definition, we have
\begin{eqnarray}
\left|\int_{-\infty}^0\Phi(-s,\theta_s\omega)u(\theta_s\omega)ds\right|_2
&\leq&\int_{-\infty}^0\left|\Phi(-s,\theta_s\omega)u(\theta_s\omega)\right|_2ds\nonumber\\
&\leq& d\int_{-\infty}^0R(\theta_s\omega){\rm e}^{-\lambda|s|}|u(\theta_s\omega)|_2ds\nonumber\\
&\leq&d\sup_{t\in\mathbb{R}}\left\{{\rm e}^{-\frac\lambda4|t|}\left|u(\theta_t\omega)\right|_2\right\}\sup_{t\in\mathbb{R}}\left\{{\rm e}^{-\frac\lambda4|t|}R(\theta_t\omega)\right\}\nonumber\\
&\cdot&\int_{-\infty}^0e^{-\frac\lambda2|s|}ds\nonumber\\
&<&\infty,\qquad\omega\in\Omega,
\nonumber\end{eqnarray}
which together with the Lebesgue's monotone convergence theorem \cite{Co} implies that
\[\lim\limits_{t\rightarrow\infty}\int_{-t}^0\Phi(-s,\theta_s\omega)u(\theta_s\omega)ds\]
exists for all $\omega\in\Omega$ and $\mathbb{R}^d_+$-valued $u$ by the order-preserving property of $\Phi$.
Furthermore, due to the boundedness of $h$ and (L), we similarly have that
$\{\varphi(t,\theta_{-t}\omega)x: t\geq0\}$ is a bounded set for all $\omega\in\Omega$ and $x\in\mathbb{R}^d_+$, which plays an important role in the subsequent sections.

\section{Measurability and behaviour of RDS generated by SDEs}
In this section, we will divide the proof of our main results into a sequence of lemmas and establish some propositions related to the measurability of the metric dynamical system $\theta$ with respect to the product $\sigma$-algebra $\mathscr{B}(\mathbb{R_-})\otimes\mathscr{F}_-$ and the dynamical behaviour of pull-back trajectories. In what follows, we may repeat some known results without proof for making our exposition self-contained. We start with definitions of future and past $\sigma$-algebras, which can be found in \cite{C,Cr}.

\begin{proposition}\label{pro1}
Define the future and the past $\sigma$-algebras for $(\theta,\Phi)$ and $(\theta,\varphi)$ as follows
\[\mathscr{F}_+^1=\sigma\{\omega\mapsto\Phi(\tau,\theta_t\omega)x:x\in\mathbb{R}^d_+,\ t,\tau\geq0\},\]
\[\mathscr{F}_-^1=\sigma\{\omega\mapsto\Phi(\tau,\theta_{-t}\omega)x:x\in\mathbb{R}^d_+,\ 0\leq\tau\leq t\},\]
\[\mathscr{F}_+^2=\sigma\{\omega\mapsto\varphi(\tau,\theta_t\omega)x:x\in\mathbb{R}^d_+,\ t,\tau\geq0\},\]
\[\mathscr{F}_-^2=\sigma\{\omega\mapsto\varphi(\tau,\theta_{-t}\omega)x:x\in\mathbb{R}^d_+,\ 0\leq\tau\leq t\}.\]
Then, we have
\begin{equation}\label{eq9}\mathscr{F}_+^1\subset\mathscr{F}_+,\qquad \mathscr{F}_-^1\subset\mathscr{F}_-,
\end{equation}
and
\begin{equation}\label{eq10}\mathscr{F}_+^2\subset\mathscr{F}_+,\qquad \mathscr{F}_-^2\subset\mathscr{F}_-.
\end{equation}
Here, $\mathscr{F}_+$ and $\mathscr{F}_-$ are defined by
\[\mathscr{F}_+=\sigma\{\omega\mapsto W_t(\omega): t\geq0\}\quad and \quad\mathscr{F}_-=\sigma\{\omega\mapsto W_t(\omega):t\leq0\}.\]
\end{proposition}
\begin{proof} We only give the proof of (\ref{eq9}), (\ref{eq10}) can be obtained analogously. By the theory of stochastic differential equations, it is clear that $\Phi(t,\omega)x$ is adapted to the filtration $\mathscr{F}_0^t=\sigma\{\omega\mapsto W_s(\omega): 0\leq s\leq t\}$, $t\geq0$, $x\in\mathbb{R}^d_+$. Consequently, for fixed $x\in\mathbb{R}^d_+$, $t,\tau\geq0$, it follows that
\begin{eqnarray}\sigma\{\omega\mapsto\Phi(\tau,\theta_t\omega)x\}
&\subset&\theta_t^{-1}\mathscr{F}_0^\tau\nonumber\\
&=&\sigma\{\omega\mapsto W_s(\theta_t\omega): 0\leq s\leq \tau\}\nonumber\\
&=&\sigma\{\omega\mapsto W_{s+t}(\omega)-W_t(\omega): 0\leq s\leq \tau\}\nonumber\\
&\subset&\mathscr{F}_+,\nonumber
\end{eqnarray}
which implies that $\mathscr{F}_+^1\subset\mathscr{F}_+$. Similarly,
for any given $x\in\mathbb{R}^d_+$, $0\leq\tau\leq t$, we have
\begin{eqnarray}\sigma\{\omega\mapsto\Phi(\tau,\theta_{-t}\omega)x\}
&\subset&\theta_{-t}^{-1}\mathscr{F}_0^\tau\nonumber\\
&=&\sigma\{\omega\mapsto W_s(\theta_{-t}\omega): 0\leq s\leq \tau\}\nonumber\\
&=&\sigma\{\omega\mapsto W_{s-t}(\omega)-W_{-t}(\omega): 0\leq s\leq \tau\}\nonumber\\
&\subset&\mathscr{F}_-,\nonumber
\end{eqnarray}
which gives that $\mathscr{F}_-^1\subset\mathscr{F}_-$, and (\ref{eq9}) is proved. The same proof works for (\ref{eq10}).
\qquad\end{proof}

By definitions of $\theta$ and the metric $\varrho$ on the space $\Omega=C_0(\mathbb{R},\mathbb{R}^d)$, it follows immediately that
$\theta:\mathbb{R}\times\Omega\mapsto\Omega$ is continuous, see \cite[Chapter 2, p. 74-75]{A}. For the purpose of readability and making
this paper self-contained, we present a proof of the continuity of $\theta(\cdot,\omega):\mathbb{R}\mapsto\Omega$ for all $\omega\in\Omega$, which is sufficient for our discussion.

\begin{proposition}\label{pro2}
For any $\omega\in\Omega$,
$\theta(\cdot,\omega):(\mathbb{R},|\cdot|)\mapsto(\Omega,\varrho)$
is continuous.
\end{proposition}
\begin{proof}
Given fixed $t_0\in\mathbb{R}$ and $\omega\in\Omega$, let $\{t_k\}_{k\in\mathbb{N}}$ be a sequence in $\mathbb{R}$ such that $t_k\rightarrow t_0$ as $k\rightarrow\infty$. We only need to show that
$\varrho(\theta_{t_k}\omega,\theta_{t_0}\omega)\rightarrow0$, i.e., $\varrho\left(\omega(t_k+\cdot)-\omega(t_k),\omega(t_0+\cdot)-\omega(t_0)\right)\rightarrow0$.
Observing that for $\forall\varepsilon>0$, there exists $N=N(\varepsilon)\in\mathbb{N}$ such that
\[\sum_{n=N}^\infty\frac{1}{2^n}\frac{\varrho_n\left(\omega(t_k+\cdot)-\omega(t_k),\omega(t_0+\cdot)-\omega(t_0)\right)}
{1+\varrho_n\left(\omega(t_k+\cdot)-\omega(t_k),\omega(t_0+\cdot)-\omega(t_0)\right)}
\leq\sum_{n=N}^\infty\frac{1}{2^n}<\varepsilon.\]
The proof is completed by showing that for all $1\leq n\leq N$, we have
\[\varrho_n\left(\omega(t_k+\cdot)-\omega(t_k),\omega(t_0+\cdot)-\omega(t_0)\right)\rightarrow0,\qquad as\quad k\rightarrow\infty.\]
Since $t_k\rightarrow t_0$, it is obvious that $\{t_k\}_{k\in\mathbb{N}}$ is bounded, which yields that there exists $M_N>0$ such that
$|t_k+t|\leq M_N$ and $|t_0+t|\leq M_N$ uniformly for all $k\in\mathbb{N}$, $t\in[-n,n]$, $1\leq n\leq N$. On the other hand, we note that any continuous function on a closed and bounded interval $[a,b]$
is uniformly continuous, which reveals that for all $1\leq n\leq N$,
\begin{eqnarray}
& &\varrho_n\left(\omega(t_k+\cdot)-\omega(t_k),\omega(t_0+\cdot)-\omega(t_0)\right)\nonumber\\
&=&\max_{t\in[-n,n]}|\omega(t_k+t)-\omega(t_k)-\omega(t_0+t)+\omega(t_0)|\nonumber\\
&\leq&\max_{t\in[-n,n]}|\omega(t_k+t)-\omega(t_0+t)|+|\omega(t_k)-\omega(t_0)|\nonumber\\
&\rightarrow&0,\quad {\rm as}\quad k\rightarrow\infty.\nonumber
\end{eqnarray}
The proof is complete.
\qquad\end{proof}

\begin{proposition}\label{pro31}$\theta:\mathbb{R_-}\times\Omega\mapsto\Omega$ is $(\mathscr{B}(\mathbb{R_-})\otimes\mathscr{F}_-, \mathscr{F}_-)$-measurable and $\theta:\mathbb{R_+}\times\Omega\mapsto\Omega$ is $(\mathscr{B}(\mathbb{R_+})\otimes\mathscr{F}_+, \mathscr{F}_+)$-measurable.
\end{proposition}
\begin{proof}The proof of this proposition is mainly motivated by the proof of Lemma 3.14 in \cite{CV}.
For convenience, we only deal with the case of time $\mathbb{R_-}$, and the rest of this proposition can be obtained analogously.
Firstly, for any $t\leq0$, we have
\begin{eqnarray}\theta_t^{-1}\mathscr{F}_-
&=&\theta_t^{-1}\sigma\{\omega\mapsto W_s(\omega):s\leq0\}\nonumber\\
&=&\sigma\{\omega\mapsto W_s(\theta_t\omega):s\leq0\}\nonumber\\
&=&\sigma\{\omega\mapsto W_{s+t}(\omega)-W_t(\omega):s\leq0\}\nonumber\\
&\subset&\mathscr{F}_-,\nonumber
\end{eqnarray}
which implies that $\theta(t,\cdot):(\Omega,\mathscr{F}_-)\mapsto(\Omega,\mathscr{F}_-)$ is measurable for any $t\in\mathbb{R_-}$.
Moreover, let $\{t_n\}_{n=1}^\infty$ denote a dense sequence in $\mathbb{R_-}$.
For any $p\geq1$, $p\in\mathbb{N}$, define $\theta_p(t,\omega)=\theta(t_n,\omega)$, where $n$ is the smallest integer such that $t$ belongs to the open interval $B(t_n,\frac1p):=\{s\in\mathbb{R_-}:|s-t_n|<\frac{1}{p}\}$. Note that $\theta_p$ is equal to the map $(t,\omega)\mapsto\theta(t_n,\omega)$ on $[B(t_n,\frac1p)-\bigcup\limits_{m<n}B(t_m,\frac1p)]\times\Omega$, and so
for any $F\in\mathscr{F}_-$, it follows that

\begin{eqnarray}
\theta^{-1}_pF&=&\bigcup^\infty\limits_{n=1}\left\{\theta^{-1}_pF
\bigcap\Big\{[B(t_n,\frac1p)-\bigcup\limits_{m<n}B(t_m,\frac1p)]\times\Omega\Big\}\right\}\nonumber\\
&=&\bigcup^\infty\limits_{n=1}\left\{[B(t_n,\frac1p)-\bigcup\limits_{m<n}B(t_m,\frac1p)]\times\theta_{t_n}^{-1}F\right\}\nonumber\\
&\in&\mathscr{B}(\mathbb{R_-})\otimes\mathscr{F}_-.\nonumber
\end{eqnarray}
This yields that $\theta_p$ is $(\mathscr{B}(\mathbb{R_-})\otimes\mathscr{F}_-, \mathscr{F}_-)$-measurable.
It is well known that $\mathscr{F}_-=\mathscr{B}_{\varrho^-}(\Omega)$, which is the Borel $\sigma$-algebra generated by open sets with respect to the pseudometric $\varrho^-$, see \cite{A,KS}. Here, the pseudometric $\varrho^-$ is defined as follows
\[\varrho^-(\omega,\omega^\ast):=\sum_{n=1}^\infty\frac{1}{2^n}\frac{\varrho_n^-(\omega,\omega^\ast)}{1+\varrho_n^-(\omega,\omega^\ast)},\quad
\varrho_n^-(\omega,\omega^\ast)=\max_{t\in[-n,0]}|\omega(t)-\omega^\ast(t)|.\]
Therefore, $\theta_p$ is $(\mathscr{B}(\mathbb{R_-})\otimes\mathscr{F}_-, \mathscr{B}_{\varrho^-}(\Omega))$-measurable for all $p\in\mathbb{N}$.
By Proposition {\rm \ref{pro2}}, it is clear that under the metric $\varrho$, $\theta_p(t,\omega)\rightarrow\theta(t,\omega)$ as $p\rightarrow\infty$ for all $t\in\mathbb{R_-}$ and $\omega\in\Omega$, which together with the fact that $\varrho^-(\omega,\omega^\ast)\leq\varrho(\omega,\omega^\ast)$ implies that under the pseudometric $\varrho^-$, $\theta_p(t,\omega)\rightarrow\theta(t,\omega)$ as $p\rightarrow\infty$ for all $t\in\mathbb{R_-}$ and $\omega\in\Omega$. Then by Theorem 21.3 in \cite{Sc}, it follows that $\theta:\mathbb{R_-}\times\Omega\mapsto\Omega$ is $(\mathscr{B}(\mathbb{R_-})\otimes\mathscr{F}_-, \mathscr{B}_{\varrho^-}(\Omega))$-measurable, i.e., $(\mathscr{B}(\mathbb{R_-})\otimes\mathscr{F}_-, \mathscr{F}_-)$-measurable. The proof is complete.
\qquad\end{proof}

\begin{proposition}\label{pro3}
For any $\mathscr{F}_-$-measurable tempered random variable $u$ in $\mathbb{R}^d_+$,
$\mathcal{K}(u)$ is a random variable with respect to the $\sigma$-algebra $\mathscr{F}_-$.
\end{proposition}
\begin{proof}
By Proposition {\rm \ref{pro31}},  it is immediate that $\theta:\mathbb{R_-}\times\Omega\mapsto\Omega$ is $(\mathscr{B}(\mathbb{R_-})\otimes\mathscr{F}_-, \mathscr{F}_-)$-measurable, which shows that $u(\theta_t\omega)$ is  $(\mathscr{B}(\mathbb{R_-})\otimes\mathscr{F}_-,\mathscr{B}(\mathbb{R}^d_+))$-measurable. Moreover, it is known that
\[(t,x)\mapsto\Phi(t,\theta_{-t}\omega)x\ \mbox{is continuous},\quad \omega\in\Omega,\]
from $\mathbb{R}_+\times\mathbb{R}^d_+$ into $\mathbb{R}^d_+$, see Remark 1.5.1 in \cite{C}. Thus, $t\mapsto\Phi(t,\theta_{-t}\omega)$ is also continuous from $\mathbb{R}_+$ into $\mathbb{R}^{d\times d}_+$, $\omega\in\Omega$. This, together with the fact that $\omega\mapsto\Phi(t,\theta_{-t}\omega)$, $t\in\mathbb{R}_+$ is $\mathscr{F}_-$-measurable by Proposition \ref{pro1}, yields that $\Phi(t,\theta_{-t}\omega)$ is
$(\mathscr{B}(\mathbb{R_+})\otimes\mathscr{F}_-,\mathscr{B}(\mathbb{R}^{d\times d}_+))$-measurable by Lemma 3.14 in \cite{CV}. Combining the definition of $\mathcal{K}$ and Fubini's theorem, it is easy to get the measurability of $\mathcal{K}(u)$. We complete the proof.
\qquad\end{proof}

\begin{proposition}\label{pro4}
For any $\tau>0$, define
\[\xi_\tau^h(\omega)=\inf\overline{\{h(\varphi(t,\theta_{-t}\omega)x):t\geq\tau\}}\]
and
\[\eta_\tau^h(\omega)=\sup\overline{\{h(\varphi(t,\theta_{-t}\omega)x):t\geq\tau\}},\quad x\in\mathbb{R}^d_+,\ \omega\in\Omega.\]
Here {\rm inf} and {\rm sup} represent the greatest lower bound and the least upper bound, respectively.
Then $\xi_\tau^h(\omega)$ and $\eta_\tau^h(\omega)$ are $\mathscr{F}_-$-measurable random variables. More precisely, they are random variables with respect to $\mathscr{F}^2_-$.
\end{proposition}
\begin{proof}
First, we claim that $\xi_\tau^h(\omega)$ and $\eta_\tau^h(\omega)$ are random variables with respect to $\mathscr{F}$. The proof is similar in spirit to Proposition 3.2 in \cite{JL}, we omit it here. Since $h(\varphi(t,\theta_{-t}\omega)x)$ is an $\mathscr{F}_-^2$-measurable random variable, $t\geq0$, $x\in\mathbb{R}^d_+$, by the virtue of the proof of Proposition 3.2 in \cite{JL}, we have $\xi_\tau^h(\omega)$ and $\eta_\tau^h(\omega)$ are $\mathscr{F}_-^2$-measurable, and due to Proposition {\rm \ref{pro1}} we can obtain the required conclusion. The proof is complete.
\qquad\end{proof}

\begin{lemma}[{\rm \cite[Lemma A.2]{FS3}}]\label{lem1}
Assume that $(x_\alpha)_{\alpha\in \Lambda}$ is a net in a normed space $X$ with a solid, normal cone $X_+\subseteq X$, which defines a partial
order on $X$. Furthermore, suppose that the net converges
to an element $x_\infty\in X$, and
\[x_\alpha^-:=\inf\{x_{\alpha'}: \alpha'\geq\alpha\}\quad and\quad x_\alpha^+:=\sup\{x_{\alpha'}: \alpha'\geq\alpha\}\]
exist for every $\alpha\in \Lambda$. Then the nets $(x_\alpha^-)_{\alpha\in \Lambda}$ and $(x_\alpha^+)_{\alpha\in \Lambda}$ also converge
to $x_\infty$.
\end{lemma}

\
\begin{lemma}\label{lem2}
Let assumptions {\rm (A)}, {\rm (L)} and {\rm (H$_1$)} hold. Then
\begin{equation}\label{eq11}\mathcal{K}(\theta-\underline{\lim}\;h(\varphi))\leq\theta-\underline{\lim}\;\varphi\leq \theta-\overline{\lim}\;\varphi\leq\mathcal{K}(\theta-\overline{\lim}\;h(\varphi))\quad {\rm for\ all}\ \omega\in\Omega,
\end{equation}
where
\[[\theta-\underline{\lim}\;h(\varphi)](\omega):=\lim_{\tau\rightarrow\infty}\xi_\tau^h(\omega)=\lim_{\tau\rightarrow\infty}\inf\{h(\varphi(t,\theta_{-t}\omega)x):t\geq\tau\},\quad x\in\mathbb{R}^d_+,\ \omega\in\Omega,\]
and
\[[\theta-\overline{\lim}\;h(\varphi)](\omega):=\lim_{\tau\rightarrow\infty}\eta_\tau^h(\omega)=\lim_{\tau\rightarrow\infty}\sup\{h(\varphi(t,\theta_{-t}\omega)x):t\geq\tau\},\quad x\in\mathbb{R}^d_+,\ \omega\in\Omega.\]
In this way, $\theta-\underline{\lim}\;\varphi$ and $\theta-\overline{\lim}\;\varphi$ can be defined similarly.
\end{lemma}
\begin{proof}
To prove (\ref{eq11}), we start with the first inequality in (\ref{eq11}). We first observe that
\[\inf\{\varphi(t,\theta_{-t}\omega)x:t\geq\tau\}=\inf\overline{\{\varphi(t,\theta_{-t}\omega)x:t\geq\tau\}},\quad x\in\mathbb{R}^d_+,\ \omega\in\Omega,\]
and
\[\inf\{h(\varphi(t,\theta_{-t}\omega)x):t\geq\tau\}=\inf\overline{\{h(\varphi(t,\theta_{-t}\omega)x):t\geq\tau\}},\quad x\in\mathbb{R}^d_+,\ \omega\in\Omega,\]
by Lemma A.1 in \cite{FS3}. By the boundedness of pull-back trajectories for (\ref{problem}) and $h$, in the same manner as the proof of
Proposition {\rm \ref{pro4}}, we can see that well-defined $\theta-\underline{\lim}\;h(\varphi)$ and $\theta-\underline{\lim}\;\varphi$ are two $(\mathscr{F}_-,
\mathscr{B}(\mathbb{R}^d_+))$-measurable random variables. Due to the boundedness of $h$,  $\theta-\underline{\lim}\;h(\varphi)$ is a tempered random variable. Hence,  $\mathcal{K}(\theta-\underline{\lim}\;h(\varphi))$ is an $(\mathscr{F}_-,
\mathscr{B}(\mathbb{R}^d_+))$-measurable random variable by Proposition {\rm \ref{pro3}}.

It follows from the definition of $\theta-\underline{\lim}\;h(\varphi)$ that
\[[\theta-\underline{\lim}\;h(\varphi)](\omega)=\lim_{\tau\rightarrow\infty}\xi_\tau^h(\omega),\quad \omega\in\Omega,\]
which together with the Lebesgue's
dominated convergence theorem \cite{Co} implies that
\[\mathcal{K}(\theta-\underline{\lim}\;h(\varphi))=\lim\limits_{\tau\rightarrow\infty}\mathcal{K}(\xi_\tau),\]
where
$\xi_\tau(\omega):=\inf\{h(\varphi(t,\theta_{-t}\omega)x):t\geq\tau\}$, $x\in\mathbb{R}^d_+,\ \omega\in\Omega$.
Now we choose an increasing sequence of time $\{\tau_n\}_{n\in\mathbb{N}}$ such that $\tau_n\uparrow\infty$ and it is sufficient to show that
\[\mathcal{K}(\xi_{\tau_n})\leq\theta-\underline{\lim}\;\varphi,\quad \omega\in\Omega,\ n\in\mathbb{N}.\]
It is evident that for all $x\in\mathbb{R}^d_+,\ \omega\in\Omega$, $\lim\limits_{t\rightarrow\infty}\Phi(t,\theta_{-t}\omega)x=0$. Therefore,
for any $\tau_n\geq0$, by the definition of $\mathcal{K}$, it follows that
\begin{eqnarray}\label{eq12}
& &[\mathcal{K}(\xi_{\tau_n})](\omega)\nonumber\\
&=& \int_{-\infty}^0\Phi(-s,\theta_s\omega)\inf\{h(\varphi(t,\theta_{-t}\bullet)x):t\geq\tau_n\}(\theta_s\omega)ds\nonumber\\
&=& \int_{-\infty}^0\Phi(-s,\theta_s\omega)\inf\{h(\varphi(t,\theta_{-t+s}\omega)x):t\geq\tau_n\}ds\nonumber\\
&=& \lim_{\substack{\tilde{t}\rightarrow\infty\\\tilde{t}\geq\tau_n}}\left\{\Phi(\tilde{t},\theta_{-\tilde{t}}\omega)x
+\int_{\tau_n-\tilde{t}}^0\Phi(-s,\theta_s\omega)\inf\{h(\varphi(t,\theta_{-t+s}\omega)x):t\geq\tau_n\}ds\right\}\nonumber\\
&=& \lim_{\substack{\tau\rightarrow\infty\\ \tau\geq\tau_n}}\inf\left\{\Phi(\tilde{t},\theta_{-\tilde{t}}\omega)x
+\int_{\tau_n-\tilde{t}}^0\Phi(-s,\theta_s\omega)\inf\{h(\varphi(t,\theta_{-t+s}\omega)x):t\geq\tau_n\}ds:\tilde{t}\geq\tau\right\}\nonumber\\
&\leq&\lim_{\substack{\tau\rightarrow\infty\\\tau\geq\tau_n}}\inf\left\{\Phi(\tilde{t},\theta_{-\tilde{t}}\omega)x
+\int_{\tau_n-\tilde{t}}^0\Phi(-s,\theta_s\omega)h(\varphi(\tilde{t}+s,\theta_{-\tilde{t}}\omega)x)ds
:\tilde{t}\geq\tau\right\}\nonumber\\
&\leq&\lim_{\tau\rightarrow\infty}\inf\left\{\Phi(\tilde{t},\theta_{-\tilde{t}}\omega)x
+\int_{-\tilde{t}}^0\Phi(-s,\theta_s\omega)h(\varphi(\tilde{t}+s,\theta_{-\tilde{t}}\omega)x)ds:\tilde{t}\geq\tau\right\}\nonumber\\
&=&[\theta-\underline{\lim}\;\varphi](\omega)\quad {\rm for\ all}\ \omega\in\Omega.\nonumber
\end{eqnarray}
Here, the above inequality follows by Lemma {\rm \ref{lem1}}, the order-preserving property of $\Phi$ and the positivity of $h$. The rest proof of (\ref{eq11}) runs as before. The proof is complete.
\qquad\end{proof}

\begin{lemma}\label{lem3}
Let assumptions {\rm (A)}, {\rm (L)} and {\rm (H$_1$)} hold. Then $h$ possesses the property
\begin{enumerate}[{\rm (i)}]
\item If $h$ is monotone, then
\begin{equation}\label{eq13}
h(\theta-\underline{\lim}\;\varphi)\leq\theta-\underline{\lim}\;h(\varphi)\leq\theta-\overline{\lim}\;h(\varphi)\leq h(\theta-\overline{\lim}\;\varphi)\quad {\rm for\ all}\ \omega\in\Omega.
\end{equation}
\item If $h$ is anti-monotone, then
\begin{equation}\label{eq14}
h(\theta-\overline{\lim}\;\varphi)\leq\theta-\underline{\lim}\;h(\varphi)\leq\theta-\overline{\lim}\;h(\varphi)\leq h(\theta-\underline{\lim}\;\varphi)\quad {\rm for\ all}\ \omega\in\Omega.
\end{equation}
\end{enumerate}
\end{lemma}
\begin{proof}
The proof of this lemma is very similar to that of Lemma 3.4 in \cite{JL}, we omit it here.
\qquad\end{proof}

\begin{lemma}\label{lem4}
Let assumptions {\rm (A)}, {\rm (L)} and {\rm (H$_1$)} hold. Then
\begin{equation}\label{eq15}
\mathcal{K}(\xi_\tau^h)\leq\theta-\underline{\lim}\;\varphi\leq \theta-\overline{\lim}\;\varphi\leq\mathcal{K}(\eta_\tau^h)\quad {\rm for\ all}\ \omega\in\Omega
\ {\rm and}\ \tau\geq0.
\end{equation}
Here $\xi_\tau^h(\omega)$ and $\eta_\tau^h(\omega)$ are defined in Proposition {\rm \ref{pro4}}. Moreover, let $\mathcal{K}^h:=h\circ\mathcal{K}$ and $\mathcal{K}^h$ is called a gain operator. Then
\begin{enumerate}[{\rm (i)}]
\item If $h$ is monotone, then for any $\tau\geq0$ and $ k\in\mathbb{N}$
\begin{equation}\label{eq16}
(\mathcal{K}^h)^k(\xi_\tau^h)\leq\theta-\underline{\lim}\;h(\varphi)\leq \theta-\overline{\lim}\;h(\varphi)\leq(\mathcal{K}^h)^k(\eta_\tau^h)\ {\rm for\ all}\ \omega\in\Omega.\
\end{equation}
\item  If $h$ is anti-monotone, then for any $\tau\geq0$ and $ k\in\mathbb{N}$
\begin{equation}\label{eq17}
(\mathcal{K}^h)^{2k}(\xi_\tau^h)\leq\theta-\underline{\lim}\;h(\varphi)\leq \theta-\overline{\lim}\;h(\varphi)\leq(\mathcal{K}^h)^{2k}(\eta_\tau^h)\ {\rm for\ all}\ \omega\in\Omega.
\end{equation}
\end{enumerate}
\end{lemma}
\begin{proof}
The proof of this lemma is very similar to that of Lemma 3.5 in \cite{JL}, we omit it here.
\qquad\end{proof}

\section{Stability and small-gain theorems}In this section, we mainly consider two kinds of output functions: one is that derivatives of $h$ are bounded and the other is that $h$ is uniformly bounded away from zero. We shall establish two small-gain theorems for guaranteeing the existence and uniqueness of positive random equilibrium and the global stability of pull-back trajectories.
\subsection{Type one: derivatives of $h$ are bounded}
In this subsection, we will state a small-gain theorem in the case that derivatives of $h$ are bounded and apply it to well-known stochastic feedback systems.
In what follows, we shall propose a natural condition that the tempered random variable $R(\omega)$ given in (L) is independent of the $\sigma$-algebra
$\mathscr{F}_-=\sigma\{\omega\mapsto W_t(\omega):t\leq0\}$. Assume that (L) holds. Then it is easy to see that the random variable $R(\omega):= {\rm sup} \{e^{\lambda t}\|\Phi(t,\omega)\|:t\geq 0\}$ is measurable with respect to the $\sigma$-algebra $\mathscr{F}_+=\sigma\{\omega\mapsto W_t(\omega): t\geq0\}$,
which is independent of $\mathscr{F}_-=\sigma\{\omega\mapsto W_t(\omega):t\leq0\}$ by the definition of the two-sided Wiener process, see \cite[Chapter 2, p. 107]{A}. This implies that such an $R$ is independent of $\mathscr{F}_-$.  Nevertheless, we still put this independence in our condition (R) because other choices of $R(\omega)$ may not be $\mathscr{F}_+$-measurable.   Moreover, we assume that the tempered random variable $R\in \mathcal{L}^1(\Omega,\mathscr{F}_+,\mathbb{P})$, which will be illustrated in our examples.
\begin{lemma}\label{lem5}
Let assumptions {\rm (A)}, {\rm (L)} and {\rm (H$_1$)} hold. Assume additionally that the following conditions on $R$ and $h$ are satisfied.
\begin{enumerate}[{\rm ({R})}]
\item
Let $R\in \mathcal{L}^1(\Omega,\mathscr{F}_+,\mathbb{P};\mathbb{R}_+)$ and be independent of the $\sigma$-algebra
$\mathscr{F}_-=\sigma\{\omega\mapsto W_t(\omega):t\leq0\}$.
\end{enumerate}
\begin{enumerate}[{\rm ({H}$_2$)}]
\item
Let $M=\max\{\sup_{x\in\mathbb{R}^d_+}|\frac{\partial h_i(x)}{\partial x_j}|, i,j=1,\ldots,d\}$ such that $\frac{Md^2\|R\|_{\mathcal{L}^1}}{\lambda}<1$, where $\|R\|_{\mathcal{L}^1}=\mathbb{E}R=\int_\Omega R(\omega)\mathbb{P}(d\omega).$ {\rm(}{\rm ({H}$_2$)} is called small-gain condition{\rm )}
\end{enumerate}
Denote by $\mathcal{L}_{\mathscr{F_-}}^1:=\mathcal{L}^1(\Omega,\mathscr{F}_-,\mathbb{P};[0,\Gamma])$ the space of all $\mathscr{F}_-$-measurable functions $f:\Omega\rightarrow[0,\Gamma]$ {\rm(}which must be integrable{\rm )}, where $\Gamma=(\Gamma_1,\ldots,\Gamma_d)$, $\Gamma_i=\sup_{x\in\mathbb{R}_+^d}|h_i(x)|$, $i=1,\ldots,d$. Then the space is complete under the metric $\|u\|_{\mathcal{L}^1}=\mathbb{E}|u|$, $u\in\mathcal{L}_{\mathscr{F_-}}^1$ and the gain operator
$\mathcal{K}^h=h\circ\mathcal{K}:(\mathcal{L}_{\mathscr{F_-}}^1,\|\cdot\|_{\mathcal{L}^1})\rightarrow(\mathcal{L}_{\mathscr{F_-}}^1,\|\cdot\|_{\mathcal{L}^1})$  is a contraction mapping.
\end{lemma}
\begin{proof}
Since $[0,\Gamma]$ is closed in $\mathbb{R}^d_+$, it is clear that $\|\cdot\|_{\mathcal{L}^1}$ defines a metric on $\mathcal{L}^1(\Omega,\mathscr{F}_-,\mathbb{P};[0,\Gamma])$ and the space endowed with this metric is complete.

We now turn to prove that $\mathcal{K}^h:\mathcal{L}_{\mathscr{F_-}}^1\rightarrow\mathcal{L}_{\mathscr{F_-}}^1$ is a contraction mapping.
Let us first point out that $\mathcal{K}^h:\mathcal{L}_{\mathscr{F_-}}^1\rightarrow\mathcal{L}_{\mathscr{F_-}}^1$ is well defined by Proposition
\ref{pro31} and Proposition \ref{pro3}, i.e., given any $\mathscr{F}_-$-measurable random variable $u$ in $\mathcal{L}_{\mathscr{F_-}}^1$,
$\mathcal{K}^h(u)$ is an $\mathscr{F}_-$-measurable random variable in $\mathcal{L}_{\mathscr{F_-}}^1$. Next, we observe that
\[\sup_{x\in\mathbb{R}^d_+}\|Dh(x)\|\leq M,\]
where $Dh(x)$ is the Jacobian of $h$.
Note that $|\Phi x|\leq d\|\Phi\|\cdot|x|$
for all $x\in\mathbb{R}^d$ and $\Phi\in\mathbb{R}^{d\times d}$, then for any $f_1, f_2\in\mathcal{L}^1(\Omega,\mathscr{F}_-,\mathbb{P};[0,\Gamma])$, we have
\begin{eqnarray}
\|\mathcal{K}^h(f_1)-\mathcal{K}^h(f_2)\|_{\mathcal{L}^1}
&=& \mathbb{E}|\int_0^1Dh[\mathcal{K}(f_2)+\mu(\mathcal{K}(f_1)-\mathcal{K}(f_2))]d\mu\cdot[\mathcal{K}(f_1)-\mathcal{K}(f_2)]|\nonumber\\
&\leq& d\sup_{x\in\mathbb{R}^d_+}\|Dh(x)\|\cdot\mathbb{E}|\mathcal{K}(f_1)-\mathcal{K}(f_2)|\nonumber\\
&\leq& Md\mathbb{E}\left|\int_{-\infty}^0\Phi(-s,\theta_s\omega)f_1(\theta_s\omega)ds
-\int_{-\infty}^0\Phi(-s,\theta_s\omega)f_2(\theta_s\omega)ds\right|\nonumber\\
&\leq& Md^2\mathbb{E}\int_{-\infty}^0\|\Phi(-s,\theta_s\omega)\|\cdot|f_1(\theta_s\omega)-f_2(\theta_s\omega)|ds\nonumber\\
&\leq& Md^2\mathbb{E}\int_{-\infty}^0{\rm e}^{\lambda s}R(\theta_s\omega)|f_1(\theta_s\omega)-f_2(\theta_s\omega)|ds\nonumber\\
&=& Md^2\int_{-\infty}^0{\rm e}^{\lambda s}\int_\Omega R(\theta_s\omega)|f_1(\theta_s\omega)-f_2(\theta_s\omega)|\mathbb{P}(d\omega)ds\nonumber\\
&=& Md^2\int_{-\infty}^0{\rm e}^{\lambda s}\int_\Omega R(\omega)|f_1(\omega)-f_2(\omega)|\mathbb{P}(d\omega)ds\nonumber\\
&=& Md^2\mathbb{E}R\cdot\mathbb{E}|f_1-f_2|\int_{-\infty}^0{\rm e}^{\lambda s}ds\nonumber\\
&=& \frac{Md^2\|R\|_{\mathcal{L}^1}}{\lambda}\|f_1-f_2\|_{\mathcal{L}^1},\nonumber
\end{eqnarray}
where the third-to-last equality holds because of $\theta_t\mathbb{P}=\mathbb{P}$, $t\in\mathbb{R}$, while the second-to-last equality has used the independence between $R$ and $\mathscr{F}_-$. The proof is complete.
\qquad\end{proof}

 Suppose that there exists a random equilibrium $v(\omega)$ such that $\lim_{t\rightarrow\infty}\varphi(t,\theta_{-t}\omega)x=v(\omega)$. Then $v$ must be $\mathscr{F}_-$-measurable. Thus, the most right choice is  to choose input space to be a subspace of $\mathscr{F}_-$-measurable space. Measurability with respect to $(\mathscr{B}(\mathbb{R_-})\otimes\mathscr{F}_-, \mathscr{F}_-)$ for $\theta:\mathbb{R_-}\times\Omega\mapsto\Omega$ (see Proposition \ref{pro31}) makes us to obtain $\mathscr{F}_-$-measurability for the input-to-state characteristic operator $\mathcal{K}(u)$ if $u$ is tempered and $\mathscr{F}_-$-measurable (see Proposition \ref{pro3}). By choosing $\mathcal{L}_{\mathscr{F_-}}^1$ to be the input space,   the gain operator $\mathcal{K}^h:\mathcal{L}_{\mathscr{F_-}}^1\rightarrow\mathcal{L}_{\mathscr{F_-}}^1$ is well defined. With the help of the small-gain condition {\rm (H$_2$)} and the independence between $R$ and the past $\sigma$-algebra $\mathscr{F}_-$, we get the contraction for the gain operator $\mathcal{K}^h$. Let $u$ be the fixed point of the gain operator $\mathcal{K}^h$. Then the image of the input-to-state characteristic operator $\mathcal{K}(u)$ for this fixed point will be a globally attracting positive random equilibrium for the pull-back flow of (\ref{problem}), which will be confirmed in the following small gain theorem.

\begin{theorem}[{\rm Small-gain Theorem I}]\label{thm1}
Let assumptions {\rm (A)}, {\rm (L)}, {\rm (H$_1$)}, {\rm (H$_2$)} and {\rm (R)} hold. Then there exists a unique fixed point $u\in\mathcal{L}^1(\Omega,\mathscr{F}_-,\mathbb{P};[0,\Gamma])$ for the gain operator
$\mathcal{K}^h$ such that for any $x\in\mathbb{R}^d_+$
\begin{equation}
\label{eq18}\lim_{t\rightarrow\infty}\varphi(t,\theta_{-t}\omega)x=[\mathcal{K}(u)](\omega)\quad \mathbb{P}\mbox{-a.s.}
\end{equation}
Furthermore, we have $\varphi(t,\omega)[\mathcal{K}(u)](\omega)=[\mathcal{K}(u)](\theta_t\omega)$, $\mathbb{P}$-a.s., $t>0$, i.e., $[\mathcal{K}(u)](\cdot)$ is a globally attracting positive random equilibrium.
\end{theorem}
\begin{proof}
Fix $\tau\geq0$. Whenever $h$ is monotone or anti-monotone, for the convenience, by
Lemma \ref{lem4}, it follows immediately that
\begin{equation}\label{eq19}(\mathcal{K}^h)^{2k}(\xi_\tau^h)\leq\theta-\underline{\lim}\;h(\varphi)\leq \theta-\overline{\lim}\;h(\varphi)\leq(\mathcal{K}^h)^{2k}(\eta_\tau^h)\quad{\rm for\ all}\ \omega\in\Omega\ {\rm and}\ k\in\mathbb{N}.
\end{equation}
Clearly, $\xi_\tau^h$ and $\eta_\tau^h$ are bounded $\mathscr{F}_-$-measurable variables in $\mathcal{L}^1(\Omega,\mathscr{F}_-,\mathbb{P};[0,\Gamma])$
by Proposition \ref{pro4}. Combining the Banach fixed point theorem \cite{Y} and Lemma \ref{lem5}, there exists a unique random variable
$u\in\mathcal{L}^1(\Omega,\mathscr{F}_-,\mathbb{P};[0,\Gamma])$ such that
\[[\mathcal{K}^h(u)](\omega)=u(\omega)\quad \mathbb{P}\mbox{-a.s.}\]
and
\begin{equation}\label{eq20}\lim_{k\rightarrow\infty}\mathbb{E}\left|(\mathcal{K}^h)^{2k}(\xi_\tau^h)-u\right|
=\lim_{k\rightarrow\infty}\mathbb{E}\left|(\mathcal{K}^h)^{2k}(\eta_\tau^h)-u\right|=0,
\end{equation}
i.e., $(\mathcal{K}^h)^{2k}(\xi_\tau^h)\xrightarrow{\mathcal{L}^1}u$ and $(\mathcal{K}^h)^{2k}(\eta_\tau^h)\xrightarrow{\mathcal{L}^1}u$,
which implies that $(\mathcal{K}^h)^{2k}(\xi_\tau^h)\xrightarrow{\mathbb{P}}u$ and $(\mathcal{K}^h)^{2k}(\eta_\tau^h)\xrightarrow{\mathbb{P}}u$.
Therefore, for any $x\in\mathbb{R}^d_+$, there exists a subsequence $\{k_j\}_{j\in\mathbb{N}}$ such that
\begin{equation}\label{eq21}\lim_{j\rightarrow\infty}[(\mathcal{K}^h)^{2k_j}(\xi_\tau^h)](\omega)
=u(\omega)=\lim_{j\rightarrow\infty}[(\mathcal{K}^h)^{2k_j}(\eta_\tau^h)](\omega)\quad
\mathbb{P}\mbox{-a.s.}
\end{equation}
The proof of (\ref{eq18}) follows by the similar arguments as in Theorem 4.2 in \cite{JL}. Finally, we show that $[\mathcal{K}(u)](\omega)>0$ for all $\omega\in\Omega$. Combining the fact that $u(\omega)=[\mathcal{K}^h(u)](\omega)>0$ for all $\omega\in\Omega$ and Proposition 6.2.2 in \cite{C}, it follows that $[\mathcal{K}(u)](\omega)>0$ for all $\omega\in\Omega$.
The proof is complete.
\qquad\end{proof}

\textsc{{\it Remark}} 2.
If $h\in C^1(\mathbb{R}^d_+,\mbox{int}\mathbb{R}^d_+)$, by Corollary 6.3.1 in \cite{C}, it is clear that $\varphi(t,\omega)(\mathbb{R}_+^d\setminus\{0\})\subset\mbox{int}\mathbb{R}_+^d$. This implies that $\varphi(t,\theta_{-t}\omega)[\mathcal{K}(u)](\theta_{-t}\omega)=[\mathcal{K}(u)](\omega)\gg0$, $\omega\in\Omega$, $t>0$. Here, $x\gg y$ means that
$x-y\in{\rm int}\mathbb{R}_+^n$ for all $x,y\in\mathbb{R}^n$.

\textsc{{\it Remark}} 3. Assume that $h\in C^1(\mathbb{R}^d_+,\mathbb{R}^d_+)$ and $h(0)=0$. Then the origin is an equilibrium for (\ref{problem}).  Theorem \ref{thm1} still holds in this case, that is, if all conditions in Theorem \ref{thm1} are satisfied except $h(x)\neq 0$, then the origin is  globally attracting. The proof is the same.

Our Theorem \ref{thm1} works for stochastic Goodwin negative feedback system, Othmer-Tyson positive feedback system and Griffith positive feedback system as well as other stochastic competitive systems with multiplicative noise. Our main task is to check  the conditions (L) and (R) in order to use Theorem \ref{thm1}, in other words, we need to choose a suitable $\lambda > 0$ and $\mathscr{F}_+$-measurable random variable  $R\in \mathcal{L}^1(\Omega,\mathscr{F}_+,\mathbb{P};\mathbb{R}_+)$ such that small-gain condition holds. During this process, the key point is to estimate the expectation of $R$.

Now we consider stochastic single loop  feedback system
\begin{equation}\label{EX4.1.0}
\left\{\begin{array}{l}
dx_1=[-\alpha_1x_1+f(x_n)]dt+\sigma_1x_1dW_t^1,\\
dx_i=[x_{i-1}-\alpha_ix_i]dt+\sigma_ix_idW_t^i,\quad 2\leq i\leq n,\end{array}\right.
\end{equation}
where $\alpha_i>0$ for $i=1,\ldots,n$ and $f\in C_b^1(\mathbb{R}_+,\mathbb{R}_+)$, i.e., $f$ and its derivative are both bounded in $\mathbb{R}_+$. Moreover, we assume that $f$ is increasing or decreasing in $\mathbb{R}_+$.  The corresponding linear homogeneous stochastic It${\rm  \hat{o}}$ type differential equations is
\begin{equation}\label{EX4.1.1}
\left\{\begin{array}{l}
dx_1=-\alpha_1x_1dt+\sigma_1x_1dW_t^1,\\
dx_i=[x_{i-1}-\alpha_ix_i]dt+\sigma_ix_idW_t^i,\quad 2\leq i\leq n.\end{array}\right.
\end{equation}
By the variation-of-constants formula, we can easily calculate the fundamental matrix $\Phi(t,\omega)$ of (\ref{EX4.1.1}) as follows.
\begin{equation}\label{EX4.1.2}
\Phi(t,\omega)=\left[\begin{array}{cccc} \Phi_{11}(t,\omega)&0&\cdots&0\\
 \Phi_{21}(t,\omega)&\Phi_{22}(t,\omega)&\cdots&0\\
 \vdots&\vdots&\ddots&\vdots\\
\Phi_{n1}(t,\omega)&\Phi_{n2}(t,\omega)&\cdots&\Phi_{nn}(t,\omega)\end{array}\right]
\end{equation}
for all $t\geq0$ and $\omega\in\Omega$, where
\begin{equation}\label{EX4.1.3}\Phi_{ii}(t,\omega)=e^{-(\alpha_i+\frac12\sigma_i^2)t+\sigma_iW_t^i(\omega)},\end{equation}
\begin{equation}\label{EX4.1.4}\Phi_{ij}(t,\omega)=\Phi_{ii}(t,\omega)\int_0^t\Phi_{ii}^{-1}(s,\omega)\Phi_{i-1,j}(s,\omega)ds,\quad 1\leq j\leq i-1,
\end{equation}
and
\[\Phi_{ij}(t,\omega)=0,\quad i+1\leq j\leq n,\]
for all $i=1,\ldots,n$.
Let $\lambda=\frac{1}{n+1}\min\{\alpha_1,\ldots,\alpha_n\}$. Then it is easy to check that
\begin{eqnarray}
\Phi_{ii}(t,\omega)\label{EX4.1.5}
&\leq&
e^{-[(n+1)\lambda+\frac12\sigma_i^2]t+\sigma_iW_t^i(\omega)}\nonumber\\
&=&
e^{-(i\lambda+\frac12\sigma_i^2)t+\sigma_iW_t^i(\omega)}e^{-(n+1-i)\lambda t}\nonumber\\
&\leq&R_i(\omega)e^{-(n+1-i)\lambda t},
\end{eqnarray}
for all $t\geq0$ and $\omega\in\Omega$, where
\begin{equation}\label{EX4.1.6}
R_i(\omega)=\sup_{t\geq0}\exp\left(-(i\lambda+\frac12\sigma_i^2)t+\sigma_iW_t^i(\omega)\right),\qquad i=1,\ldots,n.
\end{equation}
 Next, we show that $R_i(\omega)$ is a tempered random variable for all $i=1,\ldots,n$.
For any $\omega\in\Omega$ and $\gamma>0$, by (\ref{EX4.1.6}), it follows that
\begin{eqnarray}
& &\sup_{\tau\in\mathbb{R}}\left\{e^{-\gamma|\tau|}\left|R_i(\theta_\tau\omega)\right|\right\}\nonumber\\
&=&\sup_{\tau\in\mathbb{R}}\left\{e^{-\gamma|\tau|}\left[\sup_{t\geq0}\exp\left(-(i\lambda+\frac12\sigma_i^2)t+\sigma_iW_t^i(\theta_\tau\omega)\right)
\right]\right\}\nonumber\\
&=&\sup_{\tau\in\mathbb{R}}\left\{\sup_{t\geq0}\exp\left(-\gamma|\tau|-(i\lambda+\frac12\sigma_i^2)t+\sigma_iW_{t+\tau}^i(\omega)-\sigma_iW_\tau^i(\omega)\right)\right\}\nonumber\\
&\leq&\sup_{\tau\in\mathbb{R}}\left\{\sup_{t\geq0}\exp\left(-\frac{\gamma\wedge\beta_i}{2}|t+\tau|+\sigma_iW_{t+\tau}^i(\omega)\right)
\exp\left(-\frac{\gamma\wedge\beta_i}{2}|\tau|-\sigma_iW_\tau^i(\omega)\right)\right\}\nonumber\\
&\leq&\sup_{t\in\mathbb{R}}\exp\left(-\frac{\gamma\wedge\beta_i}{2}|t|+\sigma_iW_t^i(\omega)\right)
\sup_{\tau\in\mathbb{R}}\exp\left(-\frac{\gamma\wedge\beta_i}{2}|\tau|-\sigma_iW_\tau^i(\omega)\right)\nonumber\\
&<&\infty,\nonumber
\end{eqnarray}
where $\beta_i=i\lambda+\frac12\sigma_i^2$ for all $i=1,\ldots,n$ and the last inequality holds because of the law of the iterated logarithm of Brownian motions. In what follows, we claim that
\begin{equation}\label{EX4.1.7}\Phi_{ij}(t,\omega)\leq R_j(\omega)\widetilde{R}_{j+1}(\omega)\cdots\widetilde{R}_{i}(\omega)e^{-(n+1-i)\lambda t}\end{equation}
for all $1\leq j\leq i-1$, where
\begin{equation}\label{EX4.1.8}\widetilde{R}_{i}(\omega)=\int_0^\infty R_i(\theta_s\omega)e^{-\lambda s}ds,
\end{equation}
is also a tempered random variable for all $i=1,\ldots,n$. Indeed, given any $\omega\in\Omega$ and $\gamma>0$, from (\ref{EX4.1.8}),
we can see that
\begin{eqnarray}
& &\sup_{\tau\in\mathbb{R}}\left\{e^{-\gamma|\tau|}\left|\widetilde{R}_{i}(\theta_\tau\omega)\right|\right\}\nonumber\\
&=&\sup_{\tau\in\mathbb{R}}\left\{e^{-\gamma|\tau|}\int_0^\infty e^{-\lambda s}R_i(\theta_{s+\tau}\omega)ds\right\}\nonumber\\
&\leq&\sup_{\tau\in\mathbb{R}}\left\{\int_0^\infty e^{-\frac{\lambda s}{2}}e^{-\frac{\lambda\wedge\gamma}{2}|s+\tau|}R_i(\theta_{s+\tau}\omega)ds\right\}\nonumber\\
&\leq&\sup_{\tau\in\mathbb{R}}\left\{e^{-\frac{\lambda\wedge\gamma}{2}|\tau|}\left|R_i(\theta_\tau\omega)\right|\right\}\int_0^\infty e^{-\frac{\lambda s}{2}}ds\nonumber\\
&<&\infty,\nonumber
\end{eqnarray}
where the last inequality holds because of the temperedness of $R_i(\omega)$ for all $i=1,\ldots,n$.
 In order to check (\ref{EX4.1.7}), we only present the proof of $\Phi_{21}(t,\omega)$ and $\Phi_{31}(t,\omega)$, the rest can be analogously completed by induction. Combining (\ref{EX4.1.3}), (\ref{EX4.1.4}), (\ref{EX4.1.5}) and (\ref{EX4.1.6}), it is clear that
\begin{eqnarray}
\Phi_{21}(t,\omega)&=&\int_0^t\Phi_{22}(t-s,\theta_s\omega)\Phi_{11}(s,\omega)ds\nonumber\\
&\leq&
\int_0^te^{-[(n+1)\lambda+\frac12\sigma_2^2](t-s)+\sigma_2W_{t-s}^2(\theta_s\omega)}R_1(\omega)e^{-n\lambda s}ds\nonumber\\
&\leq&
\int_0^tR_2(\theta_s\omega)e^{-(n-1)\lambda(t-s)}R_1(\omega)e^{-n\lambda s}ds\nonumber\\
&\leq&e^{-(n-1)\lambda t}R_1(\omega)\int_0^\infty R_2(\theta_s\omega)e^{-\lambda s}ds\nonumber\\
&=&R_1(\omega)\widetilde{R}_{2}(\omega)e^{-(n-1)\lambda t}\nonumber
\end{eqnarray}
and
\begin{eqnarray}
\Phi_{31}(t,\omega)&=&\int_0^t\Phi_{33}(t-s,\theta_s\omega)\Phi_{21}(s,\omega)ds\nonumber\\
&\leq&
\int_0^te^{-[(n+1)\lambda+\frac12\sigma_3^2](t-s)+\sigma_3W_{t-s}^3(\theta_s\omega)}R_1(\omega)\widetilde{R}_{2}(\omega)e^{-(n-1)\lambda s}ds\nonumber\\
&\leq&
\int_0^tR_3(\theta_s\omega)e^{-(n-2)\lambda(t-s)}R_1(\omega)\widetilde{R}_{2}(\omega)e^{-(n-1)\lambda s}ds\nonumber\\
&\leq&e^{-(n-2)\lambda t}R_1(\omega)\widetilde{R}_{2}(\omega)\int_0^\infty R_3(\theta_s\omega)e^{-\lambda s}ds\nonumber\\
&=&R_1(\omega)\widetilde{R}_{2}(\omega)\widetilde{R}_{3}(\omega)e^{-(n-2)\lambda t}\nonumber
\end{eqnarray}
for all $t\geq0$ and $\omega\in\Omega$. Furthermore, we note that for all $\omega\in\Omega$, $R_i(\omega)\geq1$, which implies that $\widetilde{R}_{i}(\omega)\geq\frac1\lambda$
and
\begin{eqnarray}
\Phi_{ij}(t,\omega)&\leq&R_j(\omega)\widetilde{R}_{j+1}(\omega)\cdots\widetilde{R}_{i}(\omega)e^{-(n+1-i)\lambda t}\nonumber\\
&\leq&\lambda^{n-i}R_j(\omega)\widetilde{R}_{j+1}(\omega)\cdots\widetilde{R}_{i}(\omega)\widetilde{R}_{i+1}(\omega)\cdots\widetilde{R}_{n}(\omega)e^{-\lambda t}\nonumber\\
&\leq&\max\{1,\lambda^{n-1}\}R_j(\omega)\widetilde{R}_{j+1}(\omega)\cdots\widetilde{R}_{i}(\omega)\widetilde{R}_{i+1}(\omega)\cdots\widetilde{R}_{n}(\omega)e^{-\lambda t}\nonumber
\end{eqnarray}
for all $1\leq j\leq i-1$, $i=1,\ldots,n$. Let
\begin{equation}\label{EX4.1.9}R(\omega):=\max\{1,\lambda^{n-1}\}\bigvee_{j=1}^{n}R_j(\omega)\widetilde{R}_{j+1}(\omega)\cdots\widetilde{R}_{n}(\omega),
\end{equation}
which yields that $R(\omega)$ is tempered and
\[\|\Phi(t,\omega)\|:=\max\{|\Phi_{ij}(t,\omega)|:i,j=1,\ldots,n\}\leq R(\omega)e^{-\lambda t},\quad t\geq0,\ \omega\in\Omega.\]
Next, we will show that $R\in \mathcal{L}^1(\Omega,\mathscr{F}_+,\mathbb{P};\mathbb{R}_+)$.
In fact, for all $i=1,\ldots,n$, it is obvious that $R_i$ is $\mathscr{A}_i$-measurable, where $\mathscr{A}_i=\sigma\{\omega\mapsto W_t^i(\omega): t\geq0\}$. Consequently, $R_i$ is $\mathscr{F}_+$-measurable for all $i=1,\ldots,n$. Furthermore, for fixed $t\geq0$, $e^{-(i\lambda+\frac12\sigma_i^2)t+\sigma_i\left(W_{t+\cdot}^i(\omega)-W_\cdot^i(\omega)\right)}$ is continuous for all $\omega\in\Omega$ and
$e^{-(i\lambda+\frac12\sigma_i^2)t+\sigma_i\left(W_{t+s}^i(\cdot)-W_s^i(\cdot)\right)}$ is $\mathscr{A}_i$-measurable for all $s\geq0$ and $i=1,\ldots,n$,
which together with Lemma 3.14 in \cite{CV} implies that for any $t\geq0$
\[(s,\omega)\mapsto e^{-(i\lambda+\frac12\sigma_i^2)t+\sigma_i\left(W_{t+s}^i(\omega)-W_s^i(\omega)\right)},\quad s\geq0,\ \omega\in\Omega\]
is $(\mathscr{B}(\mathbb{R_+})\otimes\mathscr{A}_i,\mathscr{B}(\mathbb{R}_+))$-measurable. Then, by (\ref{EX4.1.6}), it follows that
\[(s,\omega)\mapsto R_i(\theta_s\omega)=\sup_{t\geq0}e^{-(i\lambda+\frac12\sigma_i^2)t+\sigma_i\left(W_{t+s}^i(\omega)-W_s^i(\omega)\right)},\quad s\geq0,\ \omega\in\Omega\]
is also $(\mathscr{B}(\mathbb{R_+})\otimes\mathscr{A}_i,\mathscr{B}(\mathbb{R}_+))$-measurable. Combining this and Fubini's theorem, it is clear that
$\widetilde{R}_i(\omega)$ is $\mathscr{A}_i$-measurable, and so is $\mathscr{F}_+$-measurable for all $i=1,\ldots,n$. The above analysis implies that
$R$ is $\mathscr{F}_+$-measurable. Now, we will prove that $R$ is $\mathcal{L}^1$-integrable. Combining the fact that an $n$-dimensional Brownian motion has $n$ independent components, (\ref{EX4.1.8}) and (\ref{EX4.1.9}), it follows that
\begin{eqnarray}\label{esti}
\mathbb{E}R&\leq&\max\{1,\lambda^{n-1}\}\sum_{j=1}^{n}\mathbb{E}\{R_j\widetilde{R}_{j+1}\cdots\widetilde{R}_{n}\}\nonumber\\
&=&\max\{1,\lambda^{n-1}\}\sum_{i=1}^{n}\frac{1}{\lambda^{n-i}}\prod_{j=i}^{n}\mathbb{E}R_j\nonumber\\
&=&\max\{1,\lambda^{n-1}\}\sum_{i=1}^{n}\frac{1}{\lambda^{n-i}}\prod_{j=i}^{n}(1+\frac{\sigma_j^2}{2j\lambda}),
\end{eqnarray}
where the last equality holds because of the property of geometric Brownian motion, i.e.,
$\mathbb{E}\sup_{t\geq0}\exp\left(-(\mu+\frac12\sigma^2)t+\sigma W_t(\omega)\right)=1+\frac{\sigma^2}{2\mu}$, where $\mu>0$ and $\sigma\in\mathbb{R}$,
see {\rm \cite[p. 585]{GP}} and {\rm \cite[p. 1639]{Pe}}.

Let $h(x)=(f(x_n),0,\ldots,0)^T$, $x\in\mathbb{R}_+^n$, $\Gamma_1=\sup_{x_n\in\mathbb{R}_+}|f(x_n)|$, $\Gamma_i=0$ for all $2\leq i\leq n$ and $M=\sup_{x_n\in\mathbb{R}_+}\frac{{\rm d}f(x_n)}{{\rm d}x_n}$. Then employing  the Small-gain Theorem I and Remark 3, we conclude the following.
\begin{corollary}\label{coro} Let $\alpha_i>0$ for $i=1,\ldots,n$ and $f\in C_b^1(\mathbb{R}_+,\mathbb{R}_+)$. Assume that $f$ is increasing or decreasing in $\mathbb{R}_+$. If
\begin{equation}\label{EX4.1.10}
\frac{Mn^2\|R\|_{\mathcal{L}^1}}{\lambda}\leq\frac{Mn^2}{\lambda}
\max\{1,\lambda^{n-1}\}\sum_{i=1}^{n}\frac{1}{\lambda^{n-i}}\prod_{j=i}^{n}(1+\frac{\sigma_j^2}{2j\lambda})<1
\end{equation}
holds, then the stochastic single loop feedback system (\ref{EX4.1.0}) admits a unique globally attracting  random equilibrium in $\mathbb{R}_+^n$.
\end{corollary}

\
{\bf Example 4.1. (Stochastic Goodwin System)} Consider $n$-dimensional {\it stochastic Goodwin negative feedback system}
\begin{equation}\label{EX4.2.1}
\left\{\begin{array}{l}
dx_1=[-\alpha_1x_1+\frac{V}{K+x_n^m}]dt+\sigma_1x_1dW_t^1,\\
dx_i=[x_{i-1}-\alpha_ix_i]dt+\sigma_ix_idW_t^i,\quad 2\leq i\leq n,\end{array}\right.
\end{equation}
where $m>1$, $K>1$, $V>0$ and $\alpha_i>0$ for $i=1,\ldots,n$. It is clear that (\ref{EX4.2.1}) is a non-monotone stochastic system, which can be regarded as the stochastic Goodwin model, see \cite{G,HTW}.
Moreover, an easy computation shows that
\[M=\sup_{x_n\in\mathbb{R}_+}\left\{\frac{mVx_n^{m-1}}{(K+x_n^m)^2}\right\}
\leq\sup_{x_n\in\mathbb{R}_+}\left\{\frac{mV(1+x_n^{m})}{(K+x_n^m)^2}\right\}
\leq\frac{mV}{K}.\]
Applying Corollary \ref{coro}, we get that if
\begin{equation}\label{EX4.1.10g}
\frac{mn^2V}{\lambda K}
\max\{1,\lambda^{n-1}\}\sum_{i=1}^{n}\frac{1}{\lambda^{n-i}}\prod_{j=i}^{n}(1+\frac{\sigma_j^2}{2j\lambda})<1
\end{equation}
is satisfied, then stochastic Goodwin negative feedback system (\ref{EX4.2.1}) possesses a globally attracting nontrivial random equilibrium.
Here, (\ref{EX4.1.10g}) holds for $V$ sufficiently small or $K$ sufficiently large. Moreover, we can have that the unique random equilibrium
$\mathcal{K}(u)$ is strongly positive, i.e., $[\mathcal{K}(u)](\omega)\gg0$ for all $\omega\in\Omega$. Noting that $f(x_n)=\frac{V}{K+x_n^m}>0$ for all $x_n\in\mathbb{R}_+$, it follows that
$u(\omega)=[\mathcal{K}^h(u)](\omega)=(u_1(\omega),0,\ldots,0)$ and $u_1(\omega)>0$ for all $\omega\in\Omega$. Combining (\ref{EX4.1.2}), (\ref{EX4.1.3})
and (\ref{EX4.1.4}), it is clear that given any $t>0$, $\Phi_{i1}(t,\omega)>0$ for all $i=1,\ldots,n$ and $\omega\in\Omega$. This together with the definition of $\mathcal{K}$ implies that $[\mathcal{K}(u)](\omega)\gg0$ for all $\omega\in\Omega$.

{\bf Example 4.2. (Stochastic Othmer-Tyson System)} Consider the following $n$-dimensional {\it stochastic Othmer-Tyson positive feedback system}:
\begin{equation}\label{EX4.3a.1}
\left\{\begin{array}{l}
dx_1=[-\alpha_1x_1+\frac{k_0(1+x_n^m)}{K+x_n^m}]dt+\sigma_1x_1dW_t^1,\\
dx_i=[x_{i-1}-\alpha_ix_i]dt+\sigma_ix_idW_t^i,\quad 2\leq i\leq n,\end{array}\right.
\end{equation}
where $k_0>0$, $K>1$, $m>1$ and $\alpha_i>0$ for $i=1,\ldots,n$. This model can be found in \cite{Ot,TO}, which is a stochastic cooperative system.
By the direct calculation, it is obvious that
\[M=\sup_{x_n\in\mathbb{R}_+}\left\{\frac{mk_0(K-1)x_n^{m-1}}{(K+x_n^m)^2}\right\}
\leq\sup_{x_n\in\mathbb{R}_+}\left\{\frac{mk_0(K-1)(1+x_n^m)}{(K+x_n^m)^2}\right\}\leq\frac{mk_0(K-1)}{K}.\]
As long as the small-gain condition
\begin{equation}\label{EX4.1.10t}
\frac{mk_0n^2(K-1)}{\lambda K}
\max\{1,\lambda^{n-1}\}\sum_{i=1}^{n}\frac{1}{\lambda^{n-i}}\prod_{j=i}^{n}(1+\frac{\sigma_j^2}{2j\lambda})<1
\end{equation}
holds, stochastic Othmer-Tyson positive feedback system (\ref{EX4.3a.1}) possesses a unique globally attracting nontrivial random equilibrium for pull-back flow by Corollary \ref{coro}. It is easy to see that the small-gain condition (\ref{EX4.1.10t}) is true for $k_0$ small enough. Furthermore, the strong positivity of the unique random equilibrium $\mathcal{K}(u)$  can be obtained by the same argument in Example 4.1.

{\bf Example 4.3. (Stochastic Griffith System)} Next, we study the following $n$-dimensional {\it stochastic Griffith positive feedback system}:
\begin{equation}\label{EX4.1.11}
\left\{\begin{array}{l}
dx_1=[-\alpha_1x_1+\frac{Kx_n^m}{1+Kx_n^m}]dt+\sigma_1x_1dW_t^1,\\
dx_i=[x_{i-1}-\alpha_ix_i]dt+\sigma_ix_idW_t^i,\quad 2\leq i\leq n,\end{array}\right.
\end{equation}
where $m>1$, $K>0$ and $\alpha_i>0$ for $i=1,\ldots,n$, see \cite{Gr}. An easy computation shows that
\begin{eqnarray}M
&=&\sup_{x_n\in\mathbb{R}_+}\left\{\frac{mKx_n^{m-1}}{(1+Kx_n^m)^2}\right\}=\frac{mKx_n^{m-1}}{(1+Kx_n^m)^2}\bigg|_{x_n^m=\frac{m-1}{K(m+1)}}\nonumber\\
&=&\frac{\sqrt[m]{K}(m+1)^2}{4m}\left(\frac{m-1}{m+1}\right)^{\frac{m-1}{m}}.\nonumber
\end{eqnarray}
If $K$ is small enough, the small-gain condition
\begin{equation}\label{EX4.1.12}
\frac{\sqrt[m]{K}n^2(m+1)^2}{4m\lambda}\left(\frac{m-1}{m+1}\right)^{\frac{m-1}{m}}
\max\{1,\lambda^{n-1}\}\sum_{i=1}^{n}\frac{1}{\lambda^{n-i}}\prod_{j=i}^{n}(1+\frac{\sigma_j^2}{2j\lambda})<1
\end{equation}
holds. Using Corollary \ref{coro} and Remark 3, it follows that the zero solution is the unique globally attracting random equilibrium for stochastic Griffith positive feedback system (\ref{EX4.1.11}).

{\bf Example 4.4.} We consider an $n$-dimensional {\it stochastic competitive system}
\begin{equation}\label{EX4.3b.1}dx_i=[-\alpha_ix_i+h_i(x)]dt+\sigma_ix_idW_t^i,\end{equation}
where $\alpha_i>0$ for all $i=1,\ldots,n$ and
\begin{equation}\label{EX4.3b.2}h_i(x):=\frac{1}{K_i+x_1^m+\ldots+x_n^m},\qquad x\in\mathbb{R}_+^n,\ i=1,\ldots,n,\end{equation}
where $m>1$ and $K_i>1$ for all $i=1,\ldots,n$.
Then, $h$ is a $C^1$-decreasing function from $\mathbb{R}_+^n$ to $\mathbb{R}_+^n\backslash\{0\}$.
It follows immediately that (\ref{EX4.3b.1}) is a stochastic competitive system.
By the direct computation, we obtain
\begin{equation}\label{EX4.3b.3}
\Phi(t,\omega)=\left[\begin{array}{cccc} \Phi_{11}(t,\omega)&0&\cdots&0\\
0&\Phi_{22}(t,\omega)&\cdots&0\\
 \vdots&\vdots&\ddots&\vdots\\
0&0&\cdots&\Phi_{nn}(t,\omega)\end{array}\right]
\end{equation}
for all $t\geq0$ and $\omega\in\Omega$, where
\begin{equation}\label{EX4.3b.4}\Phi_{ii}(t,\omega)=e^{-(\alpha_i+\frac12\sigma_i^2)t+\sigma_iW_t^i(\omega)}.\end{equation}
Consequently, it is evident that
\begin{equation}\label{EX4.3b.5}\|\Phi(t,\omega)\|:=\max\{|\Phi_{ij}(t,\omega)|:i,j=1,\ldots,n\}\leq R(\omega)e^{-\lambda t},\quad t\geq0,\ \omega\in\Omega,\end{equation}
where $\lambda=\frac12\min\{\alpha_1,\ldots,\alpha_n\}$ and
\begin{equation}\label{EX4.3b.6}
R(\omega)=\bigvee_{i=1}^{n}\sup_{t\geq0}\exp\left(-(\lambda+\frac12\sigma_i^2)t+\sigma_iW_t^i(\omega)\right).
\end{equation}
It follows that $R(\omega)$ is $\mathscr{F}_+$-measurable and independent of $\mathscr{F}_-$.
Similar to the analysis for system (\ref{EX4.1.0}), we claim that $R(\omega)$ is tempered and belongs to $ \mathcal{L}^1(\Omega,\mathscr{F}_+,\mathbb{P};\mathbb{R}_+)$. In fact, a simple calculation shows that
\begin{eqnarray}
\mathbb{E}R
&\leq&\sum_{i=1}^{n}\mathbb{E}\sup_{t\geq0}\exp\left(-(\lambda+\frac12\sigma_i^2)t+\sigma_iW_t^i(\omega)\right)\nonumber\\
&=&\sum_{i=1}^{n}(1+\frac{\sigma_i^2}{2\lambda}).\nonumber
\end{eqnarray}
Furthermore, we see that
\begin{eqnarray}M&=&\max\left\{\sup_{x\in\mathbb{R}^n_+}|\frac{\partial h_i(x)}{\partial x_j}|, i,j=1,\ldots,n\right\}\nonumber\\
&=&\max\left\{\sup_{x\in\mathbb{R}^n_+}\frac{mx_j^{m-1}}{(K_i+x_1^m+\ldots+x_n^m)^2}, i,j=1,\ldots,n\right\}\nonumber\\
&\leq&\frac{m}{K_i}\leq\frac{m}{K},\nonumber
\end{eqnarray}
where $K=\min\{K_i, i=1,\ldots,n\}$. Therefore, the small-gain condition is
\[\frac{Mn^2\|R\|_{\mathcal{L}^1}}{\lambda}
\leq\frac{mn^2}{\lambda K}\sum_{i=1}^{n}(1+\frac{\sigma_i^2}{2\lambda})<1,\]
 which must hold when $\lambda$ or $K$ is large enough.
According to the Small-gain Theorem I and Remark 2, stochastic
competitive system (\ref{EX4.3b.1}) admits a unique globally asymptotically stable positive random equilibrium, to which all the  pull-back trajectories of (\ref{EX4.3b.1}) converge.

\

\textsc{{\it Remark}} 4. Observe that our small-gain condition involves in the expectation of the random variable $R(\omega)$. So we need to provide the exact representation of $R(\omega)$ in  the definition (L) of the top Lyapunov exponent.   We point out that the choice of $\lambda$ and $R$ in the condition (L) is not unique. Usually, the bigger $\lambda$ makes $\mathbb{E}R$ bigger. As for the linear homogeneous stochastic  differential equations (\ref{EX4.1.1}), we choose $\lambda=\frac{1}{n+1}\min\{\alpha_1,\ldots,\alpha_n\}$ and $R$ as (\ref{EX4.1.9}) whose estimation is given in (\ref{esti}). We note that during this process we lose a lot. For concrete example, even if our small-gain condition (\ref{EX4.1.10}) fails, we can trace our idea by choosing other  $\lambda$ and $R$ such that small-gain theorem still holds, which is shown in the following three-dimensional {\it stochastic Othmer-Tyson positive feedback system}
\begin{equation}\label{EX4.3c.1}
\left\{\begin{array}{l}
dx_1=[-8x_1+\frac{1}{6}\cdot\frac{1+x_3^3}{\frac43+x_3^3}]dt+\frac12x_1dW_t^1,\\
dx_2=[x_1-9x_2]dt+\frac14x_2dW_t^2,\\
dx_3=[x_2-10x_3]dt+\frac13x_3dW_t^3.\end{array}\right.
\end{equation}
In fact, an easy computation shows that
\[M=\sup_{x_n\in\mathbb{R}_+}\left\{\frac{mk_0(K-1)x_n^{m-1}}{(K+x_n^m)^2}\right\}=
\frac{x_n^2}{6(\frac43+x_n^3)^2}\bigg|_{x_n^3=\frac23}
=\frac{1}{24}\left(\frac23\right)^{2/3}\]
In Corollary \ref{coro}, $\lambda = 2$ in this case. Thus
\[\max\{1,\lambda^{2}\}\sum_{i=1}^{3}\frac{1}{\lambda^{3-i}}\prod_{j=i}^{3}(1+\frac{\sigma_j^2}{2j\lambda})
\Big|_{\lambda=2}\geq2^2+2+1=7,\]
which implies that $\frac{1}{24}\left(\frac23\right)^{2/3}\times\frac92\times7>1$. That is, the small-gain condition (\ref{EX4.1.10}) does not work.
In what follows, we will prove that the small-gain condition (\ref{EX4.1.10}) can hold by changing the choice of $\lambda$ and $R$ suitably.
It is clear that
\[\Phi_{11}(t,\omega)=e^{(-8-\frac18)t+\frac12W_t^1(\omega)},\]
\[\Phi_{22}(t,\omega)=e^{(-9-\frac{1}{32})t+\frac14W_t^2(\omega)},\]
\[\Phi_{33}(t,\omega)=e^{(-10-\frac{1}{18})t+\frac13W_t^3(\omega)},\]
and
\[\Phi_{21}(t,\omega)=\int_0^te^{(-9-\frac{1}{32})(t-s)+\frac14\left(W_t^2(\omega)-W_s^2(\omega)\right)}\Phi_{11}(s,\omega)ds,\]
\[\Phi_{3i}(t,\omega)=\int_0^te^{(-10-\frac{1}{18})(t-s)+\frac13\left(W_t^3(\omega)-W_s^3(\omega)\right)}\Phi_{2i}(s,\omega)ds,\qquad i=1,2.\]
Hence, it is easy to check that
\begin{equation}\label{EX4.3c.2}\Phi_{ii}(t,\omega)\leq R_i(\omega)e^{-(4-i)t},\quad i=1,2,3,\end{equation}
for all $t\geq0$ and $\omega\in\Omega$, where
\[
R_1(\omega)=\sup_{t\geq0}\exp\left((-5-\frac18)t+\frac12W_t^1(\omega)\right),\ R_2(\omega)=\sup_{t\geq0}\exp\left((-7-\frac{1}{32})t+\frac14W_t^2(\omega)\right)\]
and
\[
R_3(\omega)=\sup_{t\geq0}\exp\left((-9-\frac{1}{18})t+\frac13W_t^3(\omega)\right).
\]
It follows that
\begin{eqnarray}
\Phi_{21}(t,\omega)
&=&\int_0^te^{(-9-\frac{1}{32})(t-s)+\frac14W_{t-s}^2(\theta_s\omega)}\Phi_{11}(s,\omega)ds\nonumber\\
&\leq&e^{-2t}R_1(\omega)\int_0^te^{-s}e^{(-7-\frac{1}{32})(t-s)+\frac14W_{t-s}^2(\theta_s\omega)}ds\nonumber\\
&\leq&e^{-2t}R_1(\omega)\int_0^\infty e^{-s}R_2(\theta_s\omega)ds\nonumber\\
&=&e^{-2t}R_1(\omega)\widetilde{R}_2(\omega),\nonumber
\end{eqnarray}
\begin{eqnarray}
\Phi_{31}(t,\omega)
&\leq&e^{-t}R_1(\omega)\widetilde{R}_2(\omega)\int_0^\infty e^{-s}R_3(\theta_s\omega)ds\nonumber\\
&=&e^{-t}R_1(\omega)\widetilde{R}_2(\omega)\widetilde{R}_3(\omega),\nonumber
\end{eqnarray}
and
\begin{eqnarray}
\Phi_{32}(t,\omega)
&\leq&e^{-t}R_2(\omega)\int_0^\infty e^{-s}R_3(\theta_s\omega)ds\nonumber\\
&=&e^{-t}R_2(\omega)\widetilde{R}_3(\omega).\nonumber
\end{eqnarray}
In order to verify the small-gain condition, we choose $\lambda=1$ and
\begin{eqnarray}
R(\omega)&=&R_1(\omega)\bigvee R_2(\omega)\bigvee R_3(\omega)
\bigvee R_1(\omega)\widetilde{R}_2(\omega)\bigvee R_1(\omega)\widetilde{R}_2(\omega)\widetilde{R}_3(\omega)\bigvee R_2(\omega)\widetilde{R}_3(\omega)\nonumber\\
&=&R_3(\omega)\bigvee R_2(\omega)\widetilde{R}_3(\omega)\bigvee R_1(\omega)\widetilde{R}_2(\omega)\widetilde{R}_3(\omega),\nonumber
\end{eqnarray}
where the last equality holds based on the fact that for all $\omega\in\Omega$, $\widetilde{R}_i(\omega)\geq1$, $i=2,3$.
Consequently, it is obvious that
\begin{eqnarray}
\mathbb{E}R
&\leq&\mathbb{E}R_3
+\mathbb{E}R_1\cdot\mathbb{E}R_2\cdot\mathbb{E}R_3
+\mathbb{E}R_2\cdot\mathbb{E}R_3\nonumber\\
&=&\frac{163}{162}+\frac{41}{40}\cdot\frac{225}{224}\cdot\frac{163}{162}+\frac{225}{224}\cdot\frac{163}{162}\nonumber\\
&<&3.0528.\nonumber
\end{eqnarray}
Therefore, we have that
\[\frac{Mn^2\|R\|_{\mathcal{L}^1}}{\lambda}
\leq\frac{1}{24}\left(\frac23\right)^{2/3}\times9\times3.0528<1.\]
That is, the small-gain condition (\ref{EX4.1.10}) holds.
This reveals that the choice of $\lambda$ and $R$ plays a key role in the proof of our result.

\
\textsc{{\it Remark}} 5. According to Chueshov {\rm \cite[p. 221]{C}},  the stochastic Othmer-Tyson positive feedback system (\ref{EX4.3a.1}) with $m=1$ is sublinear and admits a globally asymptotically attracting positive random equilibrium. As far as we know, all other results in Examples 4.1-4.4 are new.

\subsection{Type two: $h$ is uniformly bounded away from zero}
In this subsection, we will prove a small-gain theorem in the case that $h$ is uniformly bounded away from zero and present some examples.
First, we give some notations and preliminaries. Let $V$ be a real Banach space, a closed subset $V_+\subset V$ is said to be a
{\it cone} if $V_+$ is convex and $\alpha V_+\subset V_+$ for all $\alpha\in\mathbb{R}_+$, and $V_+\cap(-V_+)=\{0\}$.
We denote a partial order on $V$ by $x\leq y$ if $y-x\in V_+$, which is compatible with the structure of linear vector space $V$. A cone $V_+$ is said to be {\it solid} if it has nonempty interior points ${\rm int}V_+$. A cone $V_+$ is said to be {\it normal} if there exists a constant $c>0$ such that $\|x\|\leq c\|y\|$ whenever  $0\leq x\leq y$. Next, we will introduce definitions of {\it part} and {\it part (Birkhoff) metric}.
\begin{definition}
{\rm \bf (Part (Birkhoff) Metric)}
\begin{enumerate}[{\rm (i)}]
\item An equivalence relation is defined by $x\sim y$ if there exists $c>0$ such that $c^{-1}x\leq y\leq cx$, then
the equivalence classes on the cone $V_+$ are called the parts of $V_+$;
\item Let $C$ be any nonzero part of $V_+$. Then
\begin{equation}\label{eq22}
p(x,y):={\rm inf}\{\log c: c^{-1}x\leq y\leq cx\},\quad x,y\in C,
\end{equation}
is called the part metric (or Birkhoff metric) of $C$.
\end{enumerate}
\end{definition}
It is clear that ${\rm int}V_+$ is a part and any part is a cone in $V$.
Let $\mathcal{L}^\infty(\Omega;\mathbb{R}^d):=\mathcal{L}^\infty(\Omega,\mathscr{F},\mathbb{P};\mathbb{R}^d)$ denote the Banach space of $\mathbb{R}^d$-valued, $\mathscr{F}$-measurable, essentially bounded functions defined on $\Omega$ almost surely with the essential supremum
norm $\|f\|_\infty:=\inf\{B:|f|\leq B\ \mathbb{P}\mbox{-a.s.}\}$. The nonnegative functions in $\mathcal{L}^\infty(\Omega;\mathbb{R}^d)$ form a normal, solid cone $\mathcal{L}^\infty_+(\Omega;\mathbb{R}^d)$, see \cite[Section 1.5 and 5.2]{KLS}, where
${\rm int}\mathcal{L}^\infty_+(\Omega;\mathbb{R}^d)=
\{f: {\rm there\ extis}\ \epsilon=(\epsilon_1,\ldots,\epsilon_d)\in{\rm int}\mathbb{R}_+^d\ {\rm such\ that}
f\geq\epsilon\ \mathbb{P}\mbox{-a.s.}\}$, which consists of the family of functions essentially bounded away from zero.
To prove our main results, we start with a lemma.

\begin{lemma}\label{lem6}
Let assumptions {\rm (A)}, {\rm (L)} and {\rm (H$_1$)} hold. Assume additionally that
\begin{enumerate}[{\rm ({H}$_3$)}]
\item
$h:\mathbb{R}_+^d\rightarrow[\delta,\Gamma]\subset{\rm int}\mathbb{R}_+^d$, where $\delta=(\delta_1,\ldots,\delta_d)\gg0$.
That is, $h$ is uniformly bounded away from zero. Furthermore, we assume that
there exists a constant $T>1$ such that
$h_T$ is sublinear, i.e.,
\[\lambda h_T(x)\leq h_T(\lambda x)\quad {\rm for\ all}\ x\in\mathbb{R}_+^d\ {\rm and}\
\lambda\in[0,1],\]
where $h_T(x)=h(x)-\frac1T\delta$, $x\in\mathbb{R}_+^d$,
or there exists a constant $S>1$ such that
$h_S^{-1}$ is sublinear, i.e.,
\[\lambda h^{-1}_S(x)\leq h^{-1}_S(\lambda x)\quad {\rm for\ all}\ x\in\mathbb{R}_+^d\ {\rm and}\
\lambda\in[0,1],\]
where $h^{-1}_S(x)=(\frac{1}{h_1(x)}-\frac{1}{S\Gamma_1},\ldots,\frac{1}{h_d(x)}-\frac{1}{S\Gamma_d})$, $x\in\mathbb{R}_+^d$.
\end{enumerate}
The gain operator
$\mathcal{K}^h=h\circ\mathcal{K}:{\rm int}\mathcal{L}^\infty_+(\Omega;\mathbb{R}^d)\rightarrow{\rm int}\mathcal{L}^\infty_+(\Omega;\mathbb{R}^d)$ is defined by
\[[\mathcal{K}^h(u)](\omega)=\left[h_i\left(\int_{-\infty}^0\Phi(-s,\theta_s\omega)u(\theta_s\omega)ds\right)\right]_{i=1}^d,\quad
u\in{\rm int}\mathcal{L}^\infty_+(\Omega;\mathbb{R}^d),\]
where $u$ is the representative such that $u$ is bounded for all $\omega\in\Omega$.
Then ${\rm int}\mathcal{L}^\infty_+(\Omega;\mathbb{R}^d)$ is complete with respect to the Birkhoff metric $p$. Moreover, the gain operator
$\mathcal{K}^h=h\circ\mathcal{K}:({\rm int}\mathcal{L}^\infty_+(\Omega;\mathbb{R}^d),p)\rightarrow({\rm int}\mathcal{L}^\infty_+(\Omega;\mathbb{R}^d),p)$  is a contraction mapping.
\end{lemma}
\begin{proof}
For any $0\leq x\leq y$, $x,y\in\mathcal{L}^\infty_+(\Omega;\mathbb{R}^d)$, it is easily seen that $\|x\|_\infty\leq\|y\|_\infty$, which yields that
$\mathcal{L}^\infty_+(\Omega;\mathbb{R}^d)$ is a normal cone. Then, it follows that $({\rm int}\mathcal{L}^\infty_+(\Omega;\mathbb{R}^d),p)$ is a complete metric space, see Proposition 3.1.1 in \cite{C} or \cite{KLS}.

Now, we will show that $\mathcal{K}^h=h\circ\mathcal{K}:({\rm int}\mathcal{L}^\infty_+(\Omega;\mathbb{R}^d),p)\rightarrow({\rm int}\mathcal{L}^\infty_+(\Omega;\mathbb{R}^d),p)$ is a contraction mapping. By (H$_3$), it is clear that \[(\Gamma_1,\ldots,\Gamma_d)=\Gamma\geq[\mathcal{K}^h(u)](\omega)\geq\delta=(\delta_1,\ldots,\delta_d)\quad {\rm for\ all}\ \omega\in\Omega\ {\rm and}\ u\in{\rm int}\mathcal{L}^\infty_+(\Omega;\mathbb{R}^d).\]
That is, $\mathcal{K}^h$ is well defined from ${\rm int}\mathcal{L}^\infty_+(\Omega;\mathbb{R}^d)$ into itself.
In what follows, we will denote by $H$ the mapping
\[H(u)=\left[h_i\left(\int_{-\infty}^0\Phi(-s,\theta_s\omega)u(\theta_s\omega)ds\right)\right]_{i=1}^d,\quad
u\in{\rm int}\mathcal{L}^\infty_+(\Omega;\mathbb{R}^d),\]
and let $H^{-1}(u)=\left[\frac{1}{H_i(u)}\right]_{i=1}^d$ for all $u\in{\rm int}\mathcal{L}^\infty_+(\Omega;\mathbb{R}^d)$.
Given any $u,v\in{\rm int}\mathcal{L}^\infty_+(\Omega;\mathbb{R}^d)$, if $h_T$ is sublinear, then there exists a constant $0\leq L_1=L_1(\delta,\Gamma,T)<1$ such that
\begin{eqnarray}\label{eq23a}
p\left(H(u),H(v)\right)&=&p\left(\frac\delta T+H(u)-\frac\delta T,\frac\delta T+H(v)-\frac\delta T\right)\nonumber\\
&\leq&L_1p\left(H(u)-\frac\delta T,H(v)-\frac\delta T\right),
\end{eqnarray}
see \cite[Lemma 5.2]{FS3} or \cite[Theorem 2.6, p. 59]{N}.
Combining the definition of $H$ and the sublinearity of $h_T$, it is clear that
$[H-\frac\delta T](u)=H(u)-\frac\delta T:{\rm int}\mathcal{L}^\infty_+(\Omega;\mathbb{R}^d)\mapsto{\rm int}\mathcal{L}^\infty_+(\Omega;\mathbb{R}^d)$ is also sublinear. That is, $H-\frac\delta T$ is nonexpansive with
respect to the Birkhoff metric $p$. By (\ref{eq22}) and (\ref{eq23a}), it follows that
\begin{eqnarray}
p(\mathcal{K}^h(u),\mathcal{K}^h(v))&=&p(H(u),H(v))\nonumber\\
&\leq&L_1p\left(H(u)-\frac\delta T,H(v)-\frac\delta T\right)\nonumber\\
&\leq&L_1p\left(u,v\right)\nonumber
\end{eqnarray}
for all $u,v\in{\rm int}\mathcal{L}^\infty_+(\Omega;\mathbb{R}^d)$.
On the other hand, while $h_S^{-1}$ is sublinear, then there exists a constant $0\leq L_2=L_2(\delta,\Gamma,S)<1$ such that
\begin{eqnarray}\label{eq23b}
p\left(H(u),H(v)\right)&=&p\left(H^{-1}(u),H^{-1}(v)\right)\nonumber\\
&=&p\left(\frac1S\Gamma^{-1}+H^{-1}(u)-\frac1S\Gamma^{-1},\frac1S\Gamma^{-1}+H^{-1}(v)-\frac1S\Gamma^{-1}\right)\nonumber\\
&\leq&L_2p\left(H^{-1}(u)-\frac1S\Gamma^{-1},H^{-1}(v)-\frac1S\Gamma^{-1}\right),
\end{eqnarray}
where $\Gamma^{-1}=(\frac{1}{\Gamma_1},\ldots,\frac{1}{\Gamma_d})$, see \cite[Lemma 5.2]{FS3} or \cite[Theorem 2.6, p. 59]{N}.
Since $h_S^{-1}$ is sublinear, this guarantees the sublinearity of
$[H^{-1}-\frac1S\Gamma^{-1}](u)=H^{-1}(u)-\frac1S\Gamma^{-1}:{\rm int}\mathcal{L}^\infty_+(\Omega;\mathbb{R}^d)\mapsto{\rm int}\mathcal{L}^\infty_+(\Omega;\mathbb{R}^d)$, i.e., $H^{-1}-\frac1S\Gamma^{-1}$ is nonexpansive with
respect to the Birkhoff metric $p$. From (\ref{eq22}) and (\ref{eq23b}), it is easily seen that
\begin{eqnarray}
p(\mathcal{K}^h(u),\mathcal{K}^h(v))&=&p(H(u),H(v))\nonumber\\
&\leq&L_2p\left(H^{-1}(u)-\frac1S\Gamma^{-1},H^{-1}(v)-\frac1S\Gamma^{-1}\right)\nonumber\\
&\leq&L_2p\left(u,v\right)\nonumber
\end{eqnarray}
for all $u,v\in{\rm int}\mathcal{L}^\infty_+(\Omega;\mathbb{R}^d)$.
The proof is complete.
\qquad\end{proof}

\begin{theorem}[{\rm Small-gain Theorem II}]\label{thm2}
Let assumptions {\rm (A)}, {\rm (L)}, {\rm (H$_1$)} and {\rm (H$_3$)} hold. Then the gain operator $\mathcal{K}^h$
admits a unique fixed point $u\in{\rm int}\mathcal{L}^\infty_+(\Omega;\mathbb{R}^d)$
such that for all $x\in\mathbb{R}^d_+$
\begin{equation}
\label{eq24}\lim_{t\rightarrow\infty}\varphi(t,\theta_{-t}\omega)x=[\mathcal{K}(u)](\omega)\quad \mathbb{P}\mbox{-a.s.}
\end{equation}
Moreover, we have $\varphi(t,\omega)[\mathcal{K}(u)](\omega)=[\mathcal{K}(u)](\theta_t\omega)$, $\mathbb{P}$-a.s., $t>0$, i.e., $[\mathcal{K}(u)](\cdot)$ is a globally stable strongly positive random equilibrium.
\end{theorem}
\begin{proof}
Fix $\tau\geq0$. Without loss of generality, by Lemma \ref{lem4}, it is easy to see that
\begin{equation}\label{eq25}[\mathcal{K}^h]^{2k}(\xi_\tau^h)\leq\theta-\underline{\lim}\;h(\varphi)\leq \theta-\overline{\lim}\;h(\varphi)\leq[\mathcal{K}^h]^{2k}(\eta_\tau^h)\quad {\rm for\ all}\ \omega\in\Omega\ {\rm and}\ k\in\mathbb{N}.
\end{equation}
Observe that $h:\mathbb{R}_+^d\rightarrow[\delta,\Gamma]$ is uniformly bounded away from zero.
This implies that $\xi_\tau^h$ and $\eta_\tau^h$ both belong to ${\rm int}\mathcal{L}^\infty_+(\Omega;\mathbb{R}^d)$.
Using the Banach fixed point theorem \cite{Y}, by Lemma \ref{lem6}, there exists a unique globally attracting fixed point $u\in{\rm int}\mathcal{L}^\infty_+(\Omega;\mathbb{R}^d)$ such that
\[[\mathcal{K}^h(u)](\omega)=u(\omega)\quad \mathbb{P}\mbox{-a.s.}\]
and
\begin{equation}\label{eq26}\lim_{k\rightarrow\infty}p\left([\mathcal{K}^h]^{2k}(\xi_\tau^h),u\right)
=\lim_{k\rightarrow\infty}p\left([\mathcal{K}^h]^{2k}(\eta_\tau^h),u\right)=0.
\end{equation}
It is obvious that the norm $\|\cdot\|_\infty$ is monotone, i.e., $0\leq x\leq y$ implies that $\|x\|_\infty\leq\|y\|_\infty$.
Consequently, by Remark 3.1.1 in \cite{C}, we have
\begin{eqnarray}
\|[\mathcal{K}^h]^{2k}(\xi_\tau^h)-u\|_\infty
&\leq&\left(2e^{p\left([\mathcal{K}^h]^{2k}(\xi_\tau^h),u\right)}-e^{-p\left([\mathcal{K}^h]^{2k}(\xi_\tau^h),u\right)}-1\right)\nonumber\\
&\quad&\times\min\{\|[\mathcal{K}^h]^{2k}(\xi_\tau^h)\|_\infty,\|u\|_\infty\}\nonumber\\
&\leq&\left(2e^{p\left([\mathcal{K}^h]^{2k}(\xi_\tau^h),u\right)}-e^{-p\left([\mathcal{K}^h]^{2k}(\xi_\tau^h),u\right)}-1\right)
\cdot\|u\|_\infty\nonumber\\
&\rightarrow&0,\quad {\rm as}\quad k\rightarrow\infty,\nonumber
\end{eqnarray}
which implies that
\begin{equation}\label{eq27}\lim_{k\rightarrow\infty}[(\mathcal{K}^h)^{2k}(\xi_\tau^h)](\omega)
=u(\omega)\quad
\mathbb{P}\mbox{-a.s.}
\end{equation}
Applying the same argument, then
\begin{equation}\label{eq28}\lim_{k\rightarrow\infty}[(\mathcal{K}^h)^{2k}(\eta_\tau^h)](\omega)
=u(\omega)\quad
\mathbb{P}\mbox{-a.s.}
\end{equation}
The remainder of the proof can be handled as that of Theorem 4.2 in \cite{JL}. Furthermore, by the fact that $h$ is uniformly bounded away from zero and Remark 2, it is clear that $\mathcal{K}(u)$ is a strongly positive random equilibrium. The proof is complete.
\qquad\end{proof}

Let us now illustrate Theorem \ref{thm2} by discussing a few examples. In what follows, we will explain that our main results can be
applied to stochastic cooperative, competitive and predator-prey systems with multiplicative noise, and other non-monotone stochastic systems.
For the sake of convenience, we only present some 3-dimensional stochastic systems here.

{\bf Example 4.5.} Firstly, we consider the three-dimensional {\it stochastic cooperative system}
\begin{equation}\label{EX4.3.1}
dX_t=[AX_t+h(X_t)]dt+\sum_{i=1}^{3}G_iX_tdW_t^i\end{equation}
where
\[
A=\left[\begin{array}{ccc} -1&1&0\\
 \frac13&\frac12&0\\
0&1&-\frac13\end{array}\right],\
G_1=\left[\begin{array}{ccc} \frac32&0&0\\
 0&2&0\\
0&0&2\end{array}\right],
\
G_2=\left[\begin{array}{ccc} -3&0&0\\
 0&-2&0\\
0&0&-2\end{array}\right],
\
G_3=2I_{3\times3},
\]
and
\begin{equation}\label{EX4.3.2}h_i(x):=2+g_i(x_i),\qquad x\in\mathbb{R}_+^3,\ i=1,2,3,\end{equation}
where $g_i(x_i)=\frac{x_i}{1+x_i}$ is a $C^1$-increasing sublinear function from $\mathbb{R}_+$ to $\mathbb{R}_+$, $i=1,2,3$.
It is clear that (\ref{EX4.3.1}) is a stochastic cooperative system.
Let $\delta=(2,2,2)$ and $\Gamma=(3,3,3)$. It is easy to check that $h:\mathbb{R}_+^3\rightarrow[\delta,\Gamma]$ is an order-preserving and bounded function. Moreover, choose $T=2$, it is clear that $h_T(x)=h(x)-\frac12\delta=\frac12\delta+\tilde g(x)$ is sublinear, where $\tilde g_i(x)=g_i(x_i)$, $i=1,2,3$.
In order to use Theorem \ref{thm2}, we need to prove that the top Lyapunov exponent is negative. By Theorem 2.4.4 in \cite{C}, it follows that for any $x\in\mathbb{R}^d\setminus\{0\}$, there exists the Lyapunov exponent
\begin{equation}\label{EX4.3.3}
\lambda(x):=\lim_{t\rightarrow+\infty}\frac1t\log|\Phi(t,\omega)x|\quad{\rm for\ all}\ \omega\in\Omega^\ast,
\end{equation}
where $\Omega^\ast$ is a $\theta$-invariant set of full measure. In fact, we can choose the indistinguishable from of $\Phi(t,\omega)$ and extend the existence of Lyapunov exponents to the whole $\Omega$, see Remark 1.2.1 in \cite{C}. Moreover, it is known that $\lambda:=\max_{x\in\mathbb{R}^d\setminus\{0\}}\lambda(x)$
is the top Lyapunov exponent, see Theorem 2.4.4 and Definition 1.9.1 in \cite{C}. Therefore, in order to prove (L), it suffices to show that there exists a constant $L_\lambda>0$ such that
\begin{equation}\label{EX4.3.4}\limsup_{t\rightarrow+\infty}\frac1t\log|\Phi(t,\omega)x|\leq-L_\lambda\quad\mathbb{P}\mbox{-a.s.}\end{equation}
for all $x\in\mathbb{R}^d\setminus\{0\}$. Let us discuss the corresponding linear homogeneous stochastic It${\rm  \hat{o}}$ equations of (\ref{EX4.3.1})
\[dX_t=AX_tdt+\sum_{i=1}^{3}G_iX_tdW_t^i.\]
Hence,
\[|Ax|_2\leq\|A\|_2|x|_2=\sqrt{\frac{125}{36}}|x|_2\leq2|x|_2,\quad x\in\mathbb{R}^d,\]
\[\sum_{i=1}^{3}|G_ix|_2^2=\frac{61}{4}x_1^2+12x_2^2+12x_3^2\leq\frac{61}{4}|x|_2^2,\quad x\in\mathbb{R}^d,\]
and
\begin{eqnarray}
\sum_{i=1}^{3}|x^TG_ix|^2
&=&(\frac32x_1^2+2x_2^2+2x_3^2)^2+(-3x_1^2-2x_2^2-2x_3^2)^2+(2x_1^2+2x_2^2+2x_3^2)^2\nonumber\\
&\geq&\frac94|x|_2^4+4|x|_2^4+4|x|_2^4\nonumber\\
&=&\frac{41}{4}|x|_2^4,\quad x\in\mathbb{R}^d.
\nonumber
\end{eqnarray}
Then by Theorem 5.1 in \cite[Chapter 4]{M}, it follows easily that
\[\limsup_{t\rightarrow+\infty}\frac1t\log|\Phi(t,\omega)x|\leq-(\frac{41}{4}-2-\frac{61}{8})=-\frac58\quad\mathbb{P}\mbox{-a.s.}\]
which implies that (\ref{EX4.3.4}) holds. By the Small-gain Theorem II, stochastic cooperative system (\ref{EX4.3.1}) admits a unique
globally attracting strongly positive random equilibrium for all pull-back trajectories.

The same conclusion can be obtained if we replace $g_i(x_i)$ by $g_i(x_1+x_2+x_3)$ or let $h_i(x)=\frac{1}{2+g_i(x_i)}$, where $g_i(x_i)=\frac{1}{1+x_i}$, $i=1,2,3$.

{\bf Example 4.6.} Next, we shall study the three-dimensional {\it stochastic competitive system}
\begin{equation}\label{EX4.4.1}
dX_t=[AX_t+h(X_t)]dt+\sum_{i=1}^{3}G_iX_tdW_t^i\end{equation}
where
\[
A=\left[\begin{array}{ccc} -1&0&0\\
 0&\frac12&0\\
0&0&1\end{array}\right],\
G_1=\left[\begin{array}{ccc} 1&0&0\\
 0&\frac32&0\\
0&0&1\end{array}\right],
\
G_2=-2I_{3\times3},\
G_3=\left[\begin{array}{ccc} -\frac12&0&0\\
 0&\frac14&0\\
0&0&\frac13\end{array}\right],
\]
and
\begin{equation}\label{EX4.4.2}h_i(x):=\frac{1}{1+g_i(x_{i-1})},\qquad x\in\mathbb{R}_+^3,\ i=1,2,3,\end{equation}
where $g_i(x_{i-1})=\frac{x_{i-1}}{1+x_{i-1}}$ is a $C^1$-increasing sublinear function from $\mathbb{R}_+$ to $\mathbb{R}_+$$(x_0=x_3)$, $i=1,2,3$. It is a simple matter to see that $h:\mathbb{R}_+^3\rightarrow{\rm int}\mathbb{R}_+^3$ is an anti-order-preserving and bounded function, which yields that (\ref{EX4.4.1}) is a stochastic competitive system. Furthermore, it is easily seen that $h:\mathbb{R}_+^3\rightarrow[\delta,\Gamma]$, where $\delta=(\frac12,\frac12,\frac12)$ and $\Gamma=(1,1,1)$. Let $S=2$, it follows that
$h_S^{-1}(x)=\Gamma+\tilde g(x)-\frac12\Gamma=\frac12\Gamma+\tilde g(x)$ is sublinear, where $\tilde g_i(x)=g_i(x_{i-1})$, $i=1,2,3$.
For the purpose of using the Small-gain Theorem II, it remains to verify (L). As the analysis in Example 4.5, we are now in a position to show that
\begin{equation}\label{EX4.4.3}\limsup_{t\rightarrow+\infty}\frac1t\log|\Phi(t,\omega)x|\leq-L_\lambda\quad\mathbb{P}\mbox{-a.s.}\end{equation}
where $L_\lambda>0$ is independent of $\omega\in\Omega$ and $x\in\mathbb{R}^d\setminus\{0\}$. A simple computation gives that
\[|Ax|_2\leq\|A\|_2|x|_2=\sqrt{\frac94}|x|_2=\frac32|x|_2,\quad x\in\mathbb{R}^d,\]
\[\sum_{i=1}^{3}|G_ix|_2^2=\frac{21}{4}x_1^2+\frac{101}{16}x_2^2+\frac{46}{9}x_3^2\leq\frac{101}{16}|x|_2^2,\quad x\in\mathbb{R}^d,\]
and
\begin{eqnarray}
\sum_{i=1}^{3}|x^TG_ix|^2
&=&(x_1^2+\frac32x_2^2+x_3^2)^2+(-2x_1^2-2x_2^2-2x_3^2)^2+(-\frac12x_1^2+\frac14x_2^2+\frac13x_3^2)^2\nonumber\\
&\geq&|x|_2^4+4|x|_2^4\nonumber\\
&=&5|x|_2^4,\quad x\in\mathbb{R}^d.
\nonumber
\end{eqnarray}
This implies that
\[\limsup_{t\rightarrow+\infty}\frac1t\log|\Phi(t,\omega)x|\leq-(5-\frac{101}{32}-\frac32)=-\frac{11}{32}\quad\mathbb{P}\mbox{-a.s.}\]
by Theorem 5.1 in \cite[Chapter 4]{M}. Applying the Small-gain Theorem II, we conclude that there exists a unique globally stable strongly positive random equilibrium for stochastic competitive system (\ref{EX4.4.1}).

The same conclusion can be obtained if we replace $g_i(x_{i-1})$ by $g_i(x_1+x_2+x_3)$ or let $h_i(x)=2+\frac{1}{1+x_{i-1}^m}$, where $m>1$, $i=1,2,3$.

{\bf Example 4.7.} Finally, we investigate the three-dimensional {\it stochastic predator-prey system}
\begin{equation}\label{EX4.5.1}
dX_t=[AX_t+h(X_t)]dt+\sum_{i=1}^{3}G_iX_tdW_t^i\end{equation}
where
\[
A=\left[\begin{array}{ccc} \frac12&0&1\\
 1&-\frac13&0\\
0&1&\frac14\end{array}\right],\
G_1=3I_{3\times3},\
G_2=\left[\begin{array}{ccc} \frac32&0&0\\
 0&\frac54&0\\
0&0&1\end{array}\right],
\
G_3=\left[\begin{array}{ccc} -\frac52&0&0\\
 0&-3&0\\
0&0&-2\end{array}\right],
\]
and
\begin{equation}\label{EX4.5.2}h_i(x):=\frac{1}{3+g_i(x_{i+1})},\qquad x\in\mathbb{R}_+^3,\ i=1,2,3,\end{equation}
where $g_i(x_{i+1})=\frac{1+x_{i+1}}{2+x_{i+1}}$ is a $C^1$-increasing sublinear function from $\mathbb{R}_+$ to $\mathbb{R}_+$$(x_4=x_1)$, $i=1,2,3$.
Write $f(x)=Ax+h(x)$, $x\in\mathbb{R}_+^3$. Then $\frac{\partial f_i}{\partial x_{i-1}}(x)=1>0$ and $\frac{\partial f_{i}}{\partial x_{i+1}}(x)=-\frac{1}{(3+g_i(x_{i+1}))^2}\cdot\frac{1}{(2+x_{i+1})^2}<0$, for $i=1,2,3$ $(x_0=x_3, x_4=x_1)$. This implies that (\ref{EX4.5.1}) is a stochastic predator-prey system. Consider $\delta=(\frac14,\frac14,\frac14)$ and $\Gamma=(\frac13,\frac13,\frac13)$.
It is obvious that $h:\mathbb{R}_+^3\rightarrow[\delta,\Gamma]$ is an anti-order-preserving and bounded function.
Set $S=2$, it is evident that
$h_S^{-1}(x)=\Gamma^{-1}+\tilde g(x)-\frac12\Gamma^{-1}=\frac12\Gamma^{-1}+\tilde g(x)$ is sublinear, where $\Gamma^{-1}=(3,3,3)$ and $\tilde g_i(x)=g_i(x_{i+1})$, $i=1,2,3$.
Furthermore, we can see that
\[|Ax|_2\leq\|A\|_2|x|_2=\sqrt{\frac{493}{144}}|x|_2\leq2|x|_2,\quad x\in\mathbb{R}^d,\]
\[\sum_{i=1}^{3}|G_ix|_2^2=\frac{35}{2}x_1^2+\frac{313}{16}x_2^2+14x_3^2\leq\frac{313}{16}|x|_2^2,\quad x\in\mathbb{R}^d,\]
and
\begin{eqnarray}
\sum_{i=1}^{3}|x^TG_ix|^2
&=&(3x_1^2+3x_2^2+3x_3^2)^2+(\frac32x_1^2+\frac54x_2^2+x_3^2)^2+(-\frac52x_1^2-3x_2^2-2x_3^2)^2\nonumber\\
&\geq&9|x|_2^4+|x|_2^4+4|x|_2^4\nonumber\\
&=&14|x|_2^4,\quad x\in\mathbb{R}^d,
\nonumber
\end{eqnarray}
which together with Theorem 5.1 in \cite[Chapter 4]{M} implies that
\[\limsup_{t\rightarrow+\infty}\frac1t\log|\Phi(t,\omega)x|\leq-(14-\frac{313}{32}-2)=-\frac{71}{32}\quad\mathbb{P}\mbox{-a.s.}\]
In the same way as above, we have that (L) holds. In view of the Small-gain Theorem II, the stochastic predator-prey system (\ref{EX4.5.1}) admits a unique strongly positive random equilibrium which attracts all pull-back trajectories.

Furthermore, if we define $h_i(x)=3+\frac{1}{1+x_{i+1}^m}$, where $m>1$, $i=1,2,3$, the same result still holds.

\section{Discussion}
In this paper, we have considered the stochastic stability of nonlinear stochastic control system with inputs and outputs driven by multiplicative white noise and established two small-gain theorems in the two cases that output functions either admit bounded derivatives or are uniformly bounded away from the original. That is, there exists a unique globally attracting nonnegative random equilibrium $\mathcal{K}(u)$ for the RDS generated by those stochastic feedback systems, such as stochastic Goodwin negative feedback system,  Othmer-Tyson positive feedback system, Griffith positive feedback system and so on. Motivated by the idea in \cite{JL}, the key point in this paper is to construct a suitable complete metric space as the input space such that the gain operator $\mathcal{K}^h$ is contractive on it. In order to do this, in the case that derivatives of output functions are bounded, the joint measurability of the metric dynamical system $\theta:\mathbb{R_-}\times\Omega\mapsto\Omega$ with respect to the product $\sigma$-algebra $\mathscr{B}(\mathbb{R_-})\otimes\mathscr{F}_-$ is first established, see Proposition \ref{pro31}. This helps us to successfully receive the $\mathscr{F}_-$-measurability for the input-to-state characteristic operator $\mathcal{K}(u)$ while $u$ is tempered and $\mathscr{F}_-$-measurable, see Proposition \ref{pro3}. It is just because these measurabilities are obtained that the gain operator $\mathcal{K}^h:\mathcal{L}_{\mathscr{F_-}}^1\rightarrow\mathcal{L}_{\mathscr{F_-}}^1$ is well defined. Combining the small-gain condition {\rm (H$_2$)} and the independence between $R$ and the past $\sigma$-algebra $\mathscr{F}_-$, we finally proved that the gain operator $\mathcal{K}^h$ is a contraction mapping on the input space $\mathcal{L}_{\mathscr{F_-}}^1$. Here, the choice of the input space seems to be the best, since any globally attracting random equilibrium $v(\omega)$, i.e., $\lim_{t\rightarrow\infty}\varphi(t,\theta_{-t}\omega)x=v(\omega)$, must be $\mathscr{F}_-$-measurable.

In the use of the Small-gain Theorem I, the most important task
is to verify the small-gain condition (H$_2$) in Lemma \ref{lem5}. For this purpose, we should give suitable estimation of the upper bound for
$\frac{\|R\|_{\mathcal{L}^1}}{\lambda}$, where the positive constant $\lambda$ and the tempered random variable $R$ are defined in the condition (L). In fact, it is interesting and difficult to get the optimal upper bound of $\frac{\|R\|_{\mathcal{L}^1}}{\lambda}$ for high-dimensional stochastic control systems. Even if the exact expression of the solution is given, this is not an easy issue. For example, in the model of stochastic single loop feedback system (\ref{EX4.1.0}), our bound such as (\ref{EX4.1.10}) is conservative.

If $n=1$, then the optimal bound of $\frac{\|R\|_{\mathcal{L}^1}}{\lambda}$ can be calculated. That is, we study the following scalar SDE
\[dx=-\alpha xdt+\sigma xdW_t,\]
where $\alpha>0$ and $\sigma\neq0$. It is well known that $\Phi(t,\omega)=e^{(-\alpha-\frac{\sigma^2}{2})t+\sigma W_t(\omega)}$, $t\geq0$ and $\omega\in\Omega$.
Consequently, in order to verify the small-gain condition, we can let $0<\lambda<\alpha$ and $R(\omega)=\sup_{t\geq0}\exp\left[-(\alpha-\lambda+\frac{\sigma^2}{2})t+\sigma W_t(\omega)\right]$, $\omega\in\Omega$. This implies that
$\frac{\|R\|_{\mathcal{L}^1}}{\lambda}=\frac1\lambda+\frac{\sigma^2}{2\lambda(\alpha-\lambda)}$ and $\min_{0<\lambda<\alpha}\frac{\|R\|_{\mathcal{L}^1}}{\lambda}
=\frac{1}{\lambda_0}+\frac{\sigma^2}{2\lambda_0(\alpha-\lambda_0)}$, where $\lambda_0=\frac{(2\alpha+\sigma^2)-\sigma\sqrt{2\alpha+\sigma^2}}{2}$. Then, the small-gain condition (H$_2$) ($n=1$) can be interpreted as
$M\cdot\left[\frac{1}{\lambda_0}+\frac{\sigma^2}{2\lambda_0(\alpha-\lambda_0)}\right]<1$. However, in general ($n\geq2$), we have no good idea to get this best estimation. This will be left for future consideration. To our knowledge, our new theory provides some new insights to investigate the stochastic stability of stochastic non-monotone control systems.

\end{document}